\pdfoutput=1
\documentclass[12pt]{article}

\title{Obstructions in Symplectic Embedding Problems for Four-Dimensional Convex Toric Domains}
\author{Theodore Coyne}
\date{}

\usepackage{amssymb}
\usepackage{latexsym}
\usepackage{amsmath}
\usepackage{amsthm}
\usepackage{amscd}

\usepackage{graphicx}
\usepackage{subcaption}

\numberwithin{equation}{section}

\newtheorem{theorem}{Theorem}[section]
\newtheorem{proposition}[theorem]{Proposition}

\newtheorem{lemma}[theorem]{Lemma}
\newtheorem{lemma-definition}[theorem]{Lemma-Definition}

\theoremstyle{definition}
\newtheorem{definition}[theorem]{Definition}
\newtheorem{remark}[theorem]{Remark}

\newtheorem{example}[theorem]{Example}

\newcommand{\C}{{\mathbb C}}

\newcommand{\R}{{\mathbb R}}

\newcommand{\Z}{{\mathbb Z}}

\newcommand{\bpm}{\begin{pmatrix}}
\newcommand{\epm}{\end{pmatrix}}

\renewcommand{\epsilon}{\varepsilon}

\begin{document}

\maketitle

\begin{abstract}
In \textit{Beyond ECH Capacities}, Hutchings used the ECH cobordism map to produce obstructions for embedding one convex toric domain into another. We generalize the methods there to obstruct symplectic embeddings of a disjoint union of convex toric domains into another convex toric domain. We also find additional obstructions coming from intersection positivity and use similar methods to obstruct the existence of isotopies between two given symplectic embeddings.
\end{abstract}

\setcounter{tocdepth}{2}

\section{Introduction and Statement of Results}

In this paper, we will study symplectic embedding problems of four-dimensional toric domains. 
Let $\Omega\subset\mathbb{R}_{\geq 0}^2$ be the compact region bounded by the coordinate axes and a simple curve from $(a,0)$ to $(0,b)$ for some $a,b>0$ that only intersects the coordinate axes at these endpoints. Then, for the function $\mu:\C^2\to\R_{\geq 0}^2$ given by $\mu(z,w)=(\pi|z|^2,\pi|w|^2),$ we define the associated \textit{toric domain}

$$X_{\Omega}=\mu^{-1}(\Omega).$$

We call this curve from $(a,0)$ to $(0,b)$ the \textit{defining curve} for $\Omega$ or $X_{\Omega}.$ We equip $X_{\Omega}$ with the restriction of the standard symplectic form $\omega=\frac{i}{2}\sum_{i=1}^2dz_i\wedge d\overline{z}_i$ to make it a symplectic manifold with boundary. If the defining curve for $\Omega$ is smooth and intersects rays through the origin transversely, then $Y:=\partial X_{\Omega}$ is a contact manifold with contact form $\lambda_Y=\iota_V\omega,$ where $V=\frac{1}{2}\left(r_1\partial_{r_1}+r_2\partial_{r_2}\right)$ is the radial vector field on $\C^2$ written in polar coordinates.
When $\Omega$ is of the form $$\Omega=\{(x,y)\in\R_{\geq 0}^2:0\leq x\leq a, 0\leq y\leq f(x)\}$$ for some non-increasing and concave function $f:[0,a]\to\R_{\geq 0},$ we say that $X_{\Omega}$ is a \textit{convex toric domain} and call $f$ the \textit{defining function} of $X_{\Omega}$ or of $\Omega$. 

\begin{remark} This is a more restrictive definition of a convex toric domain than is given in \cite{Gard}, where $\Omega$ is just required to be convex, and not necessarily defined by a function.
\end{remark}

\begin{figure}[h]
\centering
\includegraphics[scale=0.1,alt={A shaded region in the first quadrant of the plane, bounded by the positive coordinate axes and the graph of a decreasing and concave function.}]{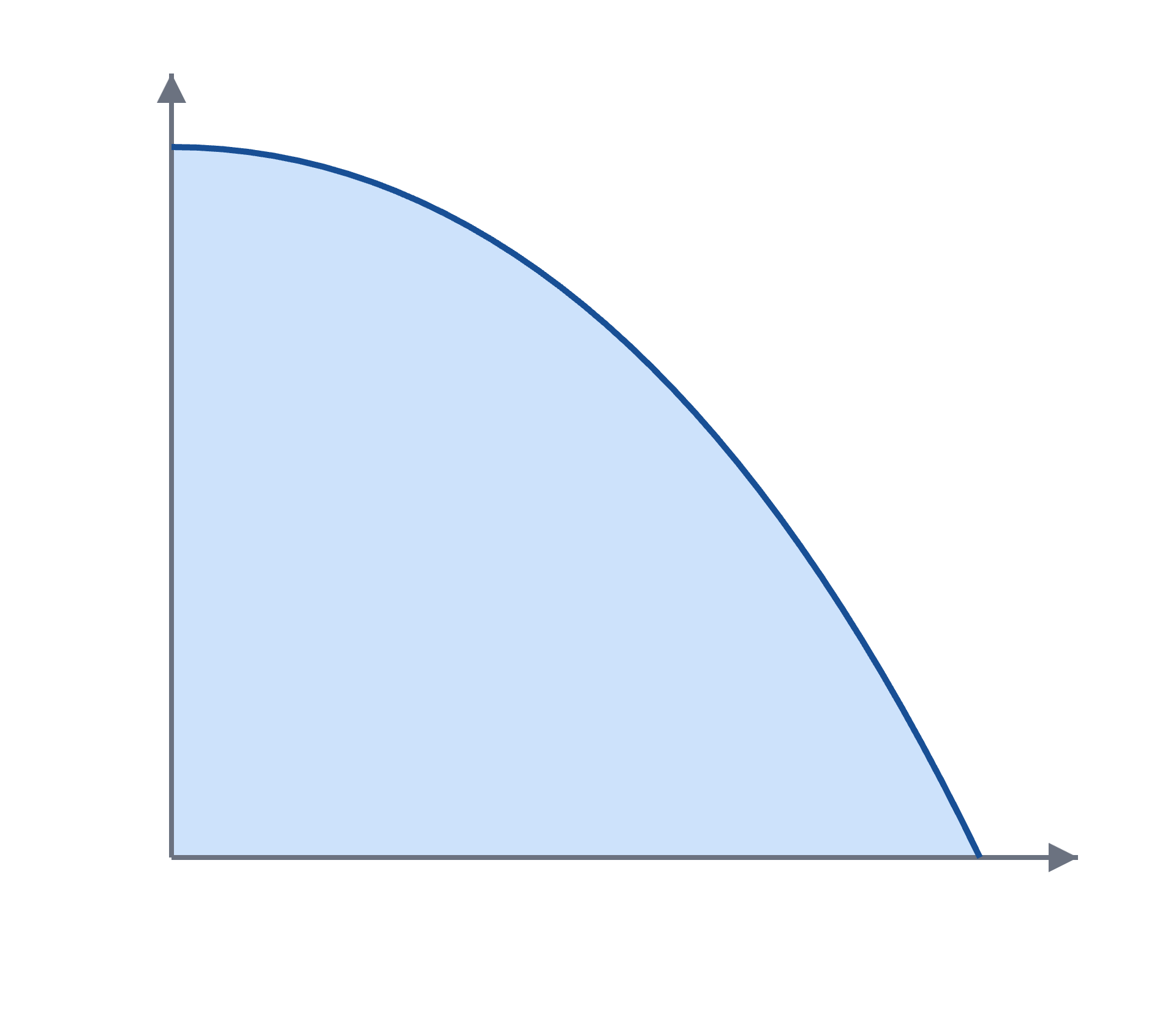}
\caption{A domain $\Omega$ defining a convex toric domain $X_{\Omega}$.}
\label{fig:convex-region}
\end{figure}
When the defining function $f$ is convex instead and $f(a)=0$, we call the corresponding region $X_{\Omega}$ a \textit{concave toric domain}.
Let $a,b>0$. Two important examples of convex toric domains we shall consider are the following:
\begin{itemize}
\item Let $\Omega$ be the rectangle with vertices $(0,0), (a,0), (0,b), \text{ and } (a,b).$ Then 
$$X_{\Omega}=\left\{(z,w)\in\C^2:\pi|z|^2\leq a \text{ and } \pi|w|^2\leq b\right\}$$

is the \textit{polydisk} $P(a,b).$ 

\item Let $\Omega$ be the triangle with vertices $(0,0), (a,0), \text{ and } (0,b).$ Then 
$$X_{\Omega}=\left\{(z,w)\in\C^2:\pi\left(\frac{|z|^2}{a}+\frac{|w|^2}{b}\right)
\leq 1\right\}$$

is the \textit{ellipsoid} $E(a,b).$ In particular, $B(c):=E(c,c)$ is the $4$-ball of radius $\sqrt{\frac{c}{\pi}}.$ We will call $c$ the \textit{capacity} of the ball.
\end{itemize}

In Section 2, we introduce some combinatorial notions relating to convex and concave toric domains that we will need to produce obstructions to symplectic embeddings.

In Section 3, we state a theorem (3.1) for obstructing symplectic embeddings of convex toric domains that relates such questions to the combinatorial objects defined in Section 2. We apply it to obtain the following obstructions to symplectically embedding disjoint unions of polydisks into another convex toric domain: 

\begin{theorem}\label{thm:two-polydisks-ball}
Let $a\in [1, 2.869].$ Suppose that there is a symplectic embedding \newline $P(a,1)\coprod P(a,1)\hookrightarrow B(c).$ Then $$c\geq 3+\left(\frac{3+\sqrt{5}}{4}-\sqrt{\frac{3(1+\sqrt{5})}{8}}\right)(a-1)\approx 3+(0.207)(a-1).$$
\end{theorem}

\begin{remark}
When $a=1,$ this bound is sharp, as two copies of $P(1,1)$ do in fact symplectically embed into $B(c)$ for any $c>3,$ as illustrated below.
\end{remark}

\begin{figure}[h]
\centering
\includegraphics[scale=0.5,alt={A red square to the left of a blue square. The square meet along a common edge. They sit inside of a triangle, representing an embedding of a disjoint union of two square polydisks into a ball}]{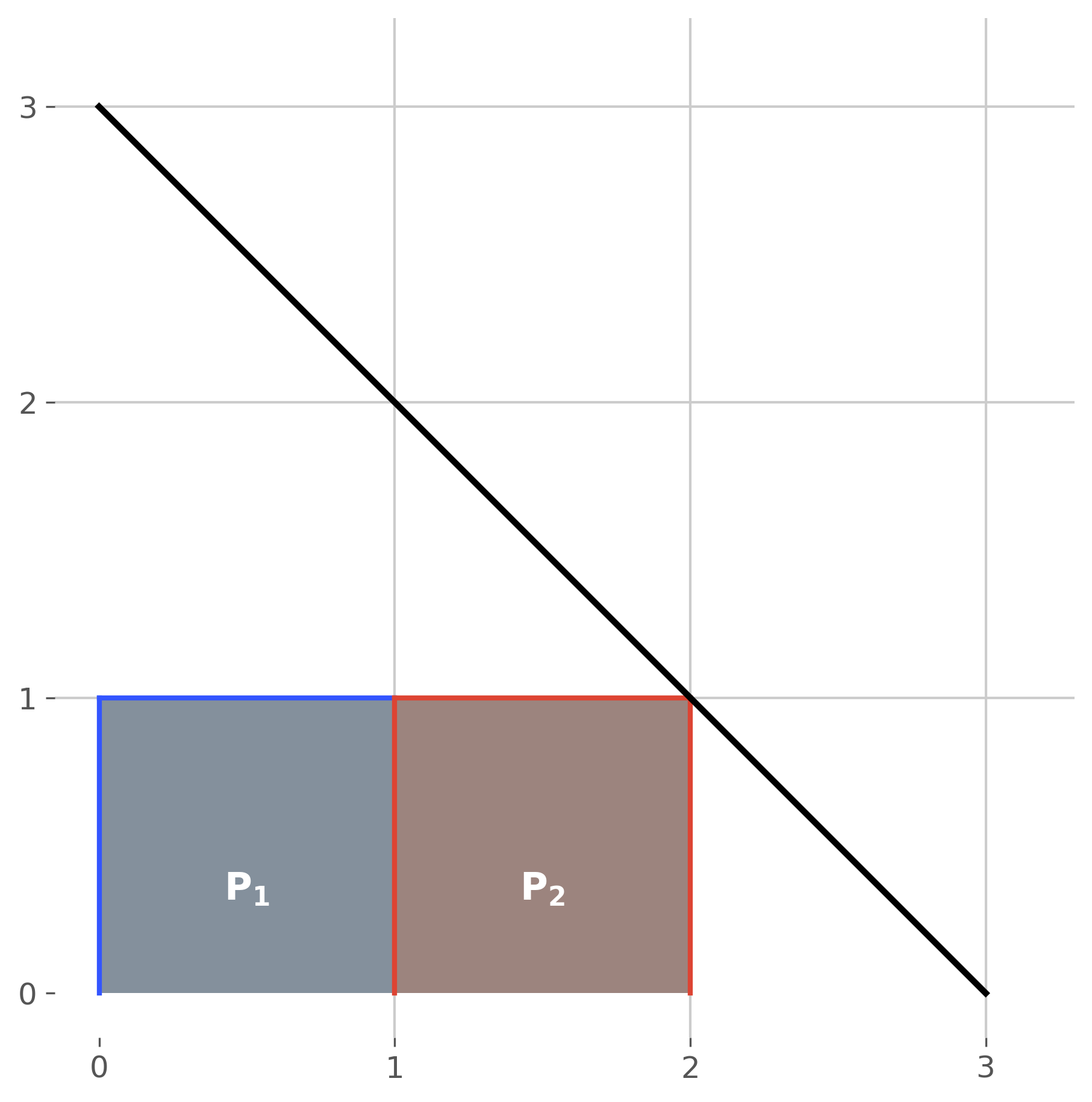}
\caption[An embedding of a disjoint union of two square polydisks into a ball.]{\small Here the preimage $\mu^{-1}(P_1)$ is the standard square polydisk $P(1,1)$ and the preimage $\mu^{-1}(P_2)$ is the product of the disk of area $1$ with an annulus of area $1$, and hence any polydisk smaller than $P(1,1)$ can be embedded into it (since a disk can be embedded into an annulus of any larger area via an area-preserving map).}
\label{fig:polydisk-remark}
\end{figure}

\begin{remark} In the proof of this theorem, we give a lower bound for $c$ that is valid for all $a\geq 1$. For large values of $a$, the bound is weaker than the volume constraint \newline $\text{vol}(B(c))\geq 2 \text{vol}(P(a,1))$. The interval $[1,2.869]$ is (a numerical approximation of) where our bound is stronger than the volume constraint. Other numerical constants throughout arise similarly.
\end{remark}

\begin{theorem}\label{thm:two-polydisks-polydisk}
Let $a\in [1,2.348].$ If $P(a,1)\coprod P(a,1)$ symplectically embeds into $P(c,c)$, then $$c\geq \begin{cases} 2 & 1\leq a\leq 2\\ 2+\left(\frac{23-\sqrt{401}}{24}\right)(a-1) & 2\leq a\leq 2.348.\end{cases}$$
\end{theorem}

\begin{remark} For $1\leq a\leq 2,$ this bound is sharp by a similar moment map picture as for the previous result (and follows trivially from the case $a=1$).
\end{remark}

We also state a similar theorem for concave toric domains (Theorem \ref{thm:main-concave}) in this section. 

In Section 4, we review some facts about embedded contact homology and how they relate to convex and concave toric domains.

In Section 5, we prove Theorem \ref{thm:main-convex} and the analogous theorem for concave domains (Theorem \ref{thm:main-concave}) by applying the ECH cobordism map to the cobordism determined by a symplectic embedding.

In Section 6, we use similar techniques to obstruct the existence of isotopies between two given symplectic embeddings. Obstructions obtained this way may be sharp, as in the following:
\begin{theorem}\label{thm:reflection-isotopy}
Let $r:\C^2\to\C^2$ be the reflection $r(z,w)=(w,z).$ Suppose that $X_{\Omega^-}$ and $X_{\Omega^+}$ are convex toric domains such that $X_{\Omega^-}\cup \left(r\left(X_{\Omega^-}\right)\right)\subseteq X_{\Omega^+}$. Let $\iota:X_{\Omega^-}\hookrightarrow X_{\Omega^+}$ be the inclusion map. Then $\iota$ and $r\circ\iota,$ considered as symplectic embeddings of $X_{\Omega^-}$ into $X_{\Omega^+}$, are isotopic by a path of symplectic embeddings if and only if there exists a positive real number $R$ such that $X_{\Omega^-}\subseteq B(R)\subseteq X_{\Omega^+}.$ 
\end{theorem}

We conclude with an appendix, where we see how intersection positivity can put more constraints on $J$-holomorphic curves that may exist in a completed cobordism. In particular, we see that certain moduli spaces $\mathcal{M}(\alpha,\beta)$ often only contain a single element satisfying some mild conditions. We also see that many possibilities for $J$-holomorphic curves that are not obstructed by Theorem \ref{thm:main-convex} may nonetheless be ruled out. 

\section*{Aknowledgements}
The author would like to thank his thesis advisor Michael Hutchings for countless helpful conversations that made this work possible.

\section{Convex and Concave Generators}
In this section, we review some combinatorial definitions introduced in \cite{Beyond} relating to convex toric domains that we will need to study the symplectic embedding problems described in the previous section. We also introduce analogous notions for concave toric domains.

\begin{definition}\label{def:convex-integral-path}
A \textit{convex integral path} is a simple and piecewise linear curve $\Lambda$ in the first quadrant $\R_{\geq 0}^2$ from a point $(0,y(\Lambda))$ to a point $(x(\Lambda),0)$ such that 
\begin{itemize}
\item The above endpoints and all vertices of $\Lambda$ lie in $\Z_{\geq 0}^2.$
\item The plane region bounded by the coordinate axes and $\Lambda$ is convex.
\item Each edge of $\Lambda$ (segment between two consecutive lattice points) is either vertical or has nonpositive slope. 
\end{itemize}
\end{definition}

\begin{definition}\label{def:convex-generator}
A \textit{convex generator} is a convex integral path where each edge is labeled either $``e"$ or $``h"$, and no two edges that have the same slope may both be labeled $``h"$, and no vertical or horizontal edges may be labeled $``h".$ Two convex generators are considered equivalent when they have the same underlying convex integral path, and the labelings are the same up to permutations among the edges with a fixed slope. Convex generators will only be considered modulo this equivalence.

By associating to each factor $e_{ab}$ or $h_{ab}$ an edge vector $(a,-b)$ of the same labeling, we may identify a convex generator with a commutative formal product of terms $e_{ab}$ and $h_{ab}$, where no factor $h_{ab}$ may be repeated, and no factor $h_{10}$ or $h_{01}$ may appear. Note that in either case, $a$ and $b$ are coprime by our definition of an edge of a convex integral path. Given two convex generators with no $h_{ab}$ factor in common, we define their product to be the convex generator obtained considering them as products of $e_{ab}$ and $h_{ab}$ terms and formally multiplying these products.
\end{definition}

Convex generators are related to convex toric domains through the following notion.

\begin{definition}\label{def:convex-action}
If $\Lambda=e_{ab}$ or $h_{ab}$ and $\Omega\subset\R_{\geq 0}^2$ defines a convex toric domain $X_{\Omega}$ as in Section 1, define the \textit{action} $\mathcal{A}_{\Omega}(\Lambda)$ as the value of $c$ for which the line $bx+ay=c$ is externally tangent to $\Omega.$ For general $\Lambda$, the action $\mathcal{A}_{\Omega}(\Lambda)$ is the sum of the actions of each $e_{ab}$ and $h_{ab}$ appearing in $\Lambda$ (with multiplicity).
\end{definition}

\begin{example} For any $\Omega$ defining a convex toric domain $X_{\Omega}$, we have $\mathcal{A}_{\Omega}(e_{10})=y(\Omega)$ and $\mathcal{A}_{\Omega}(e_{01})=x(\Omega),$ where $(x(\Omega),0)$ and $(0,y(\Omega))$ are respectively where the defining curve for $\Omega$ meets the $x$- and $y$-axes.
\end{example}

We also need the notion of the index of a convex generator. 
\begin{definition}\label{def:convex-index} The \textit{index} of a convex generator $\Lambda$ is defined as 
$$I(\Lambda)=2(L(\Lambda)-1)-h(\Lambda),$$
where $L(\Lambda)$ is the number of lattice points (including on the boundary) in the plane region bounded by the coordinate axes and the path $\Lambda$, and $h(\Lambda)$ is the number of edges of $\Lambda$ that are labeled $``h".$
\end{definition}

\begin{definition}\label{def:convex-minimal}
A convex generator $\Lambda$ is \textit{minimal} for $\Omega$ if all edges of $\Lambda$ are labeled $``e"$ and if $\Lambda'\neq\Lambda$ is a convex generator with all edges labeled $``e"$ and $I(\Lambda')=I(\Lambda)$, then $\mathcal{A}_{\Omega}(\Lambda')>\mathcal{A}_{\Omega}(\Lambda).$ 
\end{definition}

We shall also need analogous notions for concave domains.

\begin{definition}\label{def:concave-integral-path}A \textit{concave integral path} is a path $\Lambda$ as in Definition \ref{def:convex-integral-path}, except that:
\begin{itemize}
\item No edges of $\Lambda$ are horizontal or vertical.
\item The compact region bounded by the coordinate axes and $\Lambda$ has convex complement (in $\R_{\geq 0}^2)$ instead of being convex itself.
\end{itemize}
\end{definition}

\begin{definition}\label{def:concave-generator} A concave generator is a concave integral path where each edge is labeled either $``e"$ or $``h"$ with the same restrictions and notion of equivalence as in Definition \ref{def:convex-generator} (except that the condition that no vertical or horizontal edges may be labeled $``h"$ is now vacuous).
\end{definition}

\begin{definition}\label{def:concave-action} When $\Omega$ defines a concave toric domain and $\Lambda$ is a concave generator, the action $\mathcal{A}_{\Omega}(\Lambda)$ is defined the same as for the convex case, except that ``externally tangent'' is replaced by ``internally tangent''. We may identify a concave generator with a commutative formal product of $h_{ab}$ and $e_{ab}$ terms and define products of concave generators as in the convex case.
\end{definition}

\begin{definition}\label{def:concave-index} When $\Lambda$ is a concave generator, its index is defined as 

$$I(\Lambda)=2L^{\text{int}}(\Lambda)+h(\Lambda),$$

where $L^{\text{int}}(\Lambda)$ is the number of lattice points in the plane region bounded by $\Lambda$ and the coordinate axes that do not lie on $\Lambda$, and $h(\Lambda)$ has the same meaning as in the convex case. 
\end{definition}

\begin{definition}\label{def:concave-maximal} We say that a concave generator $\Lambda$ is \textit{maximal} for $\Omega$ when it uniquely maximizes action for $\Omega$ among all concave generators with the same index and all edges labeled $``e"$.
\end{definition}

Finally, we have some notions that will be useful for both convex and concave generators.

\begin{definition}\label{def:multiplicity} For either a concave or a convex generator $\Lambda$, we define its \textit{multiplicity} $m(\Lambda)$ to be its number of edges. Equivalently, when we identify $\Lambda$ with a formal product of $e_{ab}$ and $h_{ab}$ terms, $m(\Lambda)$ is the sum of the exponents of each of these factors that appear.
\end{definition}

\begin{definition}\label{def:are}
When $\Lambda$ is a concave or convex integral path or generator, define $\text{Area}(\Lambda)$ to be the area of the compact region in $\R_{\geq 0}^2$ bounded by $\Lambda$ and the coordinate axes.
\end{definition}

\section{Calculations using Convex and Concave Generators}
In this section, we will state and use a theorem that relates the existence of a symplectic embedding of a disjoint union of convex toric domains into another convex toric domain to combinatorial questions about convex generators. The following is a generalization of Theorem 1.19 in \cite{LectureNotes} to such disconnected domains.

\begin{theorem}\label{thm:main-convex}
Let $X_{\Omega_i}$ for $i=1,\ldots, m$ and $X_{\Omega'}$ be convex toric domains. Suppose that there is a symplectic embedding $\varphi:\coprod_{i=1}^m X_{\Omega_i}\hookrightarrow X_{\Omega'}.$ Let $\Lambda'$ be a convex generator which is minimal for $X_{\Omega'}.$ Let $I_1,\ldots,I_m$ be nonnegative even integers such that $I_1+\ldots+I_m=I(\Lambda').$ Then there exist convex generators $\Lambda^1,\ldots,\Lambda^m$  and factorizations $\Lambda'=\Lambda'_1\ldots\Lambda_n'$ and $\Lambda^i=\Lambda_1^i\ldots\Lambda_n^i$ for each $i$ such that the following conditions hold:

\begin{itemize}
\item{(a)} For each $i\in\{1,\ldots, m\},$ we have $I(\Lambda^i)=I_i.$
\item{(b)} For each $j\in\{1,\ldots,n\}$, we have $\sum_{i=1}^m I(\Lambda_j^i)=I(\Lambda'_j).$
\item{(c)} For each subset $S\subseteq\{1,\ldots,n\}$, we have $\sum_{i=1}^m I(\prod_{j\in S}\Lambda_j^i)=I(\prod_{j\in S}\Lambda_j').$
\item{(d)} For each $j\in\{1,\ldots,n\}$, we have $\sum_{i=1}^m \mathcal{A}_{\Omega_i}(\Lambda_j^i)\leq \mathcal{A}_{\Omega'}(\Lambda_j').$
\item{(e)} For all $j\in\{1,\ldots,n\}$, we have $\sum_{i=1}^m x(\Lambda_j^i)+y(\Lambda_j^i)-\frac{1}{2}h(\Lambda_j^i)\geq x(\Lambda_j')+y(\Lambda_j')+m(\Lambda_j')-1.$ We will refer to this inequality as the ``genus bound''.
\item{(f)} Given $k,j\in\{1,\ldots,n\}$ if $\Lambda_k'\neq\Lambda_j'$ or if $\Lambda_k^i\neq\Lambda_j^i$ for any $i\in\{1,\ldots,m\},$ then $\Lambda_k^i$ and $\Lambda_j^i$ have no elliptic factors $e_{ab}$ in common for any $i\in\{1,\ldots,m\}.$
\item{(g)} For all $j\in\{1,\ldots,n\}$ and all $i\in\{1,\ldots,m\}$, the following inequality holds whenever $I(\Lambda'_j)>2$: $$\text{min}(x(\Lambda'_j),y(\Lambda'_j))\geq\text{min}(x(\Lambda^i_j),y(\Lambda^i_j)).$$
\end{itemize}
\end{theorem}

\begin{remark}
The condition ``$\Lambda'$ is minimal for $X_{\Omega'}$'' can be weakened to ``all edges of $\Lambda'$ are labeled $`e'$''. We only use the condition that $\Lambda'$ is minimal to guarantee that it represents a nonzero ECH class. That this is true when all edges are labeled $``e"$ follows from Conjecture A.3 in \cite{Beyond}, which is proved as Corollary 4.15 in \cite{Choi}.
\end{remark}

To keep track of the data given by Theorem \ref{thm:main-convex}, it is helpful to consider the matrix of convex generators $$\begin{pmatrix}\Lambda'_1 & \Lambda'_2 & \ldots & \Lambda'_n\\ \Lambda^1_1 & \Lambda^1_2 & \ldots & \Lambda^1_n \\ \vdots & \vdots & \ & \vdots \\ \Lambda^m_1 & \Lambda^m_2 & \ldots & \Lambda^m_n\end{pmatrix}.$$ As we will see in the proof of the theorem, each column corresponds to a connected component of an index zero $J$-holomorphic current $\mathcal{C}$, and each row corresponds to the ends of $\mathcal{C}$ at a given boundary $\partial X_{\Omega'}$ or $\partial X_{\Omega_i}.$ Given a collection of columns \\ $C_{j_1}=\begin{pmatrix}\Lambda'_{j_1}\\ \vdots \\ \Lambda^m_{j_1}\end{pmatrix},\ldots,C_{j_k}=\begin{pmatrix}\Lambda'_{j_k}\\ \vdots \\ \Lambda^m_{j_k}\end{pmatrix},$ we may consider the product of these columns: $$\begin{pmatrix}\tilde{\Lambda}'\\ \tilde{\Lambda}^1 \\ \vdots \\ \tilde{\Lambda}^m\end{pmatrix}:=\begin{pmatrix}\Lambda'_{j_1}\ldots\Lambda'_{j_k}\\ \Lambda^1_{j_1}\ldots\Lambda^1_{j_k}\\ \vdots \\ \Lambda^m_{j_1}\ldots\Lambda^m_{j_k}\end{pmatrix}.$$

We will refer to this operation as ``merging'' columns. 

Item (c) in Theorem \ref{thm:main-convex} gives an index equality for this merged column: 

$$\sum_{i=1}^m I(\tilde{\Lambda}^i)=I(\tilde{\Lambda}').$$

By additivity, it also follows from Theorem \ref{thm:main-convex} that this new column still obeys the action inequality: $$\sum_{i=1}^m\mathcal{A}_{\Omega_i}(\tilde{\Lambda}^i)\leq\mathcal{A}_{\Omega'}(\tilde{\Lambda}').$$

We also note that the inequality in (e) applied to the merged column reads:

$$\sum_{i=1}^m x(\tilde{\Lambda}^i)+y(\tilde{\Lambda}^i)-\frac{1}{2}h(\tilde{\Lambda}^i)\geq x(\tilde{\Lambda}')+y(\tilde{\Lambda}')+m(\tilde{\Lambda}')-k.$$

We shall use Theorem \ref{thm:main-convex} to obstruct certain symplectic embeddings. We will generally accomplish this by taking $\Lambda'$ to be a power of a fixed convex generator $\Lambda'=(\Lambda'_0)^d$ and take a limit as $d\to\infty$. For this to yield useful information, we will need to exclude the case that some column is repeated arbitrarily many times in this limit. 

\begin{lemma}\label{lem:repeated-column}
Suppose that a column $C=\begin{pmatrix}\Lambda'\\ \Lambda^1 \\ \vdots \\ \Lambda^m\end{pmatrix}$ occurs (at least) twice in the matrix of convex generators described in the previous section with $I(\Lambda')>0$. Then one of the $\Lambda^i$ equals $e_{ab}$ for some nonnegative and coprime $a,b$ and the remaining generators $\Lambda^k$ for $k\neq i$ are empty. We also have the relations:
\begin{itemize}
\item $a+b=x(\Lambda')+y(\Lambda')+m(\Lambda')-1$ and
\item $ab=2\text{Area}(\Lambda').$
\end{itemize}
\end{lemma}

\begin{remark}
A similar restriction for repeated convex generators appears in Proposition 3.4 in \cite{CN} to obstruct embeddings of a (single) polydisk $P(a,1)$ into a ball $B(c).$
\end{remark}

\begin{proof}
For any convex generator $\Lambda$, an application of Pick's Theorem yields

$$L(\Lambda^k)=\text{Area}(\Lambda)k^2+\frac{1}{2}(x(\Lambda)+y(\Lambda)+m(\Lambda))k+1.$$

\begin{itemize}
\item Example 1: $\Lambda=e_{10}e_{01}$ has area $1$ and $x(\Lambda)+y(\Lambda)+m(\Lambda)=4,$ and this formula correctly predicts $L((e_{01}e_{10})^k)=k^2+2k+1=(k+1)^2.$
\item Example 2: $\Lambda=e_{11}$ has area $1/2$ and $x(\Lambda)+y(\Lambda)+m(\Lambda)=3,$ so this formula predicts $L(e_{11}^k)=\frac{1}{2}k^2+\frac{3}{2}k+1,$ which correctly matches a count of $1+2+\ldots+(k+1).$ 
\end{itemize}

\begin{remark}
A polynomial counting the number of lattice points in dilates of a lattice polytope is known as an Ehrhart polynomial. When the polytope is only assumed to have rational vertices, the counting function is called an Ehrhart quasipolynomial instead. See \cite{Ehrhart} for examples of these being applied to obstruct symplectic embeddings of an ellipsoid into another ellipsoid.
\end{remark}

Since the column $C$ is repeated, none of the convex generators considered may contain any edges labeled $``h"$. Merging two columns equal to $C$ gives the following equality for $k=1$ and $k=2$:

\begin{align*}
\sum_{i=1}^m I((\Lambda^i)^k)=&2\left(\sum_{i=1}^m\text{Area}(\Lambda^i)\right)k^2+\left(\sum_{i=1}^m x(\Lambda^i)+y(\Lambda^i)+m(\Lambda^i)\right)k\\
=&2\text{Area}(\Lambda')k^2+(x(\Lambda')+y(\Lambda')+m(\Lambda'))k\\
=& I\left((\Lambda')^k\right).
\end{align*}

Since both middle expressions are quadratic in $k$ with constant terms of zero, being equal for $k=1$ and $k=2$ guarantees that they are equal for all $k.$ In particular, their linear coefficients are equal:

$$\sum_{i=1}^m x(\Lambda^i)+y(\Lambda^i)+m(\Lambda^i)=x(\Lambda')+y(\Lambda')+m(\Lambda').$$

We combine with the genus bound $$\sum_{i=1}^m x(\Lambda^i)+y(\Lambda^i)\geq x(\Lambda')+y(\Lambda')+m(\Lambda')-1$$ to find

$$\sum_{i=1}^m m(\Lambda^i)\leq 1.$$

It follows that one of the $\Lambda^i$ has $m(\Lambda^i)=1$ and the remaining $\Lambda^k$ for $k\neq i$ are empty. So $\Lambda^i=e_{ab}$ for some coprime and nonnegative integers $a$ and $b$, where the relations in the statement of the lemma come from equating the linear and quadratic coefficients above.
\end{proof}

We will also need to know how the index changes when multiplying two convex generators.

\begin{lemma}\label{lem:index-shift}
Let $\Lambda$ be a convex generator. Let $c_{ab}(\Lambda)$ be the value for which the line \newline$bx+ay=c_{ab}(\Lambda)$ is externally tangent to $\Lambda$. Then $$I(e_{ab}\Lambda)=I(e_{ab})+I(\Lambda)+2c_{ab}(\Lambda).$$
\end{lemma}

\begin{proof}
As in the proof of the previous lemma, an application of Pick's Theorem yields $$I(\Lambda)=2\text{Area}(\Lambda)+x(\Lambda)+y(\Lambda)+m(\Lambda)-h(\Lambda).$$ All terms except the area term are additive, so $$I(e_{ab}\Lambda)=I(e_{ab})+I(\Lambda)+2\left(\text{Area}(e_{ab}\Lambda)-\text{Area}(\Lambda)-\text{Area}(e_{ab})\right).$$

That $c_{ab}(\Lambda)$ is the difference between $\text{Area}(e_{ab}\Lambda)$ and $\text{Area}(e_{ab})+\text{Area}(\Lambda)$ may be seen by subdividing the plane region bounded by the axes and $e_{ab}\Lambda$ into pieces.
\end{proof}

We are now ready to apply Theorem \ref{thm:main-convex} to symplectic embedding problems.

\begin{proof}[Proof of Theorem \ref{thm:two-polydisks-ball}]
We apply Theorem \ref{thm:main-convex} to $\Lambda'=e_{11}^{d'}$, which may be seen to be minimal for $B(c),$ and choose the index split as $I(\Lambda^1)=I(\Lambda^2)=\frac{d'(d'+3)}{2}$, with $d'$ being chosen such that $d'(d'+3)$ is a multiple of $4$.

Let $e_{11}^{d'}=e_{11}^{d_1}\ldots e_{11}^{d_n}$ be a factorization given by Theorem \ref{thm:main-convex}. First, we check that as $d'\to\infty$, there must be some column with $d_i$ arbitrarily large.

Suppose for contradiction that $d_i\leq K$ always. There are only finitely many possible columns, so some must be repeated arbitrarily many times as $d'\to\infty.$ By Lemma \ref{lem:repeated-column}, there must be a repeated column either of the form $\begin{pmatrix}e_{11}^d\\ \emptyset \\ e_{ab}\end{pmatrix}$ or of the form $\begin{pmatrix}e_{11}^d \\ e_{ab} \\ \emptyset\end{pmatrix}.$ Also, by the chosen index split $I(\Lambda^1)=I(\Lambda^2)=\frac{d'(d'+3)}{2},$ there must be repeated columns of both forms. However, when we merge two columns, one of each form, the index of the negative generators is additive on merging, but $I(e_{11}^{d_1+d_2})>I(e_{11}^{d_1})+I(e_{11}^{d_2})$. 

We now divide our problem into two cases. In either case, we may assume that $x(\Lambda^i)$ and $y(\Lambda^i)$ are bounded above by a fixed constant multiple of $d$, as otherwise the action inequality $$dc\geq\sum_{i=1}^2 x(\Lambda^i)+ay(\Lambda^i)$$ gives a stronger bound for $c$ than the statement of our theorem.

Case 1: There exists a repeated column $C_1$ of the form $\begin{pmatrix}e_{11}^r\\ \emptyset \\ e_{mn}\end{pmatrix}$ or of the form $\begin{pmatrix}e_{11}^r \\ e_{mn} \\ \emptyset\end{pmatrix}.$ Without loss of generality, take the second case. Note that not both of these cases can occur by index considerations as above.

Let $C_2=\begin{pmatrix}e_{11}^d\\ \Lambda^1 \\ \Lambda^2 \end{pmatrix}$ be a column where $d$ may be taken to be arbitrarily large. Since every repeated column is empty in the component of $\Lambda^2$, there are such columns where $I(\Lambda^2)\geq I(\Lambda^1).$ In particular, $I(\Lambda^2)\geq\frac{d^2+3d}{2}.$

We immediately have the following constraints from Theorem \ref{thm:main-convex}.

The index constraint (a) yields
$$d(d+3)=I(\Lambda^1)+I(\Lambda^2)\leq 2(x(\Lambda^1)+1)(y(\Lambda^1)+1)+2(x(\Lambda^2)+1)(y(\Lambda^2)+1).$$

By the assumption that $x(\Lambda^i)$ and $y(\Lambda^i)$ are $O(d)$, the righthand side is $$2\sum_{i=1}^2 x(\Lambda^i)y(\Lambda^i)+O(d),$$

Also, the genus bound gives

$$\sum_{i=1}^2x(\Lambda^i)+y(\Lambda^i)\geq 3d-1.$$

By merging $C_1$ and $C_2$ and applying the index constraint to the merged column, we find that 
\begin{align*}
2c_{mn}(\Lambda^1)&=I(e_{mn}\Lambda^1)-I(e_{mn})-I(\Lambda^1)\\ &=I(e_{11}^{d+r})-I(e_{11}^r)-I(e_{11}^d)\\ &=2dr
\end{align*} 
by Lemma \ref{lem:index-shift}, and so the line $nx+my=dr$ is externally tangent to $\Lambda^1.$ In particular, the point $(x(\Lambda^1),y(\Lambda^1))$ lies on or above this line by convexity, giving an additional constraint $\frac{n}{r}x(\Lambda_1)+\frac{m}{r}y(\Lambda_1)\geq d.$

Hence, we may bound 
\begin{align*}
d&\leq \frac{n}{r}x(\Lambda^1)+\frac{m}{r}y(\Lambda^1)\\&\leq\max\left(\frac{n}{r},\frac{m}{r}\right)(x(\Lambda^1)+y(\Lambda^1))\\&< \left(\frac{3+\sqrt{5}}{2}\right)(x(\Lambda^1)+y(\Lambda^1))
\end{align*}

using the relations $mn=r^2$ and $m+n=3r-1$ from Lemma \ref{lem:repeated-column}, and solving for $m/r$ and $n/r$ explicitly.

Finally, we use the action inequality $$dc\geq (x(\Lambda^1)+x(\Lambda^2))+a(y(\Lambda^1)+y(\Lambda^2)).$$

Putting $X_i=x(\Lambda^i)/d$ and similarly for $Y_i$, we keep only the leading terms in $d$ and take $d\to\infty$ to find

\begin{align*}c&\geq\text{min}\left\{X_1+aY_1+X_2+aY_2\bigg\vert \sum_{i=1}^2 X_i+Y_i\geq 3, \sum_{i=1}^2 X_iY_i\geq \frac{1}{2}, X_1+Y_1\geq\frac{2}{3+\sqrt{5}},X_2Y_2\geq\frac{1}{4}\right\}\\&=\begin{cases} 3+\left(\frac{3+\sqrt{5}}{4}-\sqrt{\frac{3(1+\sqrt{5})}{8}}\right)(a-1) & 1\leq a\leq a^* \\\sqrt{2a}+\frac{3-\sqrt{5}}{2} & a\geq a^*,\end{cases}\end{align*}
where $a^*=\frac{(5+3\sqrt{5})+\sqrt{66+30\sqrt{5}}}{2}\approx 11.62.$
\begin{remark}
Note that for large $a$, this bound is weaker than the volume constraint $c\geq 2\sqrt{a}.$ This is to be expected because for large values of $a,$ the volume bound may be recovered by merging all columns, using an even index split, and ignoring the genus bound. For large $a$, the index bound is stronger than the genus bound, and in order to use the genus bound as above, we could not guarantee an equal index split in the column considered. 

For $1\leq a\leq 2.869,$ this bound is sharper than the volume bound.
\end{remark}

Case 2: There do not exist any repeated columns. In this case, for any fixed $\varepsilon>0$, there are only finitely many columns of the form $\begin{pmatrix}e_{11}^d\\ \Lambda^1 \\ \Lambda^2\end{pmatrix}$ with $d\leq1/\varepsilon.$ Let $N=N(\varepsilon)$ be an upper bound for the number of such columns.

We merge all columns with $d>1/\varepsilon$ in the above notation, resulting in a column of the form $C_1=\begin{pmatrix}e_{11}^r\\ \Lambda^1 \\ \Lambda^2\end{pmatrix}$ with $d'-r<N/\varepsilon$ since at most $N$ columns are excluded from the merge, each with exponent of $e_{11}$ smaller than $1/\varepsilon$ in the first entry. Also, let $C_2=\begin{pmatrix}e_{11}^s\\ \tilde{\Lambda}^1 \\ \tilde{\Lambda}^2\end{pmatrix}$ be the merge of the remaining columns, so $s<N/\varepsilon.$ Let $p,q>0$  so that $e_{pq}$ contains the $\tilde{\Lambda_i}$ as subsets. Here $p,q$ are chosen to be independent of $d'$, but may depend on $\varepsilon.$ Note that $c_{pq}(\Lambda^i)$ is bounded by a constant (depending on $\varepsilon$) multiple of $d'$ by the assumption on $x(\Lambda^i)$ and $y(\Lambda^i).$

We may estimate 

\begin{align*}
\frac{d'^2+3d'}{2}=I(\Lambda^i\tilde{\Lambda^i})\leq I(e_{pq}\Lambda^i)=I(e_{pq})+I(\Lambda^i)+2c_{pq}(\Lambda^i)
\end{align*}

to conclude that the difference $\frac{d'^2}{2}-I(\Lambda^i)$ is bounded above by a constant multiple of $d'.$

So, in this case, after taking the limit $d'\to\infty$, we no longer have the constraint \newline $X_1+Y_1\geq\frac{2}{3+\sqrt{5}}$, but we obtain a new constraint that $X_iY_i\geq\frac{1}{4}$ for $i=1,2.$

Also, the action inequality for $C_1$ reads

\begin{align*}
rc&\geq x(\Lambda^1)+ay(\Lambda^1)+x(\Lambda^2)+ay(\Lambda^2)
\\ c&\geq\left(\frac{1}{d'+N/\varepsilon}\right)\left(x(\Lambda^1)+ay(\Lambda^1)+x(\Lambda^2)+ay(\Lambda^2)\right).
\end{align*}

Dividing through all constraints by leading powers of $d'$ and taking $d'$ to be sufficiently large compared to $\varepsilon$, we obtain

\begin{align*}c\geq &\text{min}\left\{X_1+aY_1+X_2+aY_2\bigg\vert \sum_{i=1}^2 X_i+Y_i\geq 3, X_iY_i\geq \frac{1}{4}\right\}\\
&=\begin{cases} 3+\left(\frac{3-\sqrt{3}}{2}\right)(a-1) & 1\leq a\leq \frac{7+3\sqrt{5}}{2}\\ 2\sqrt{a}: & a\geq \frac{7+3\sqrt{5}}{2}.\end{cases}
\end{align*}

This bound for $c$ is always larger than the bound we found in case $1,$ so the result follows.

\end{proof}

This calculation can be repeated for embedding three copies of $P(a,1)$ into $B(c)$, in which case we find that

\begin{proposition}
Let $a\in[1,1.96]$. If $\coprod_{i=1}^3 P(a,1)$ symplectically embeds into $B(c)$, then 
$$c\geq 3+\left(\frac{(1+\sqrt{2})(r-\sqrt{r^2-4})}{2\sqrt{6}}\right)(a-1),$$

where $r=\frac{3\sqrt{6}+\sqrt{30}}{2+2\sqrt{2}}.$

\end{proposition}

As in the previous bounds derived, this is still sharp at $a=1$.

We more generally have a sharp obstruction for embedding a union of square polydisks into certain ellipsoids.

\begin{proposition}\label{prop:square-polydisk-ellipsoid}
Let $b$ be a positive integer. There exists a symplectic embedding \newline $P(1,1)\coprod P(1,1)\hookrightarrow E(bc,c)$ if $bc>b+2$ and there does not when $bc<b+2.$
\end{proposition}

\begin{proof}
If $bc>b+2$, then the embedding can be given by inclusion in a straightforward way, as this is the condition for the point $(2,1)$ to lie below the line $x+by=bc$ in $\R^2.$
To prove that $bc\geq b+2$ is a necessary condition, we apply Theorem \ref{thm:main-convex} to the convex generator $\Lambda'=e_{b1}^{d'},$ with $d'$ a positive integer. As in the case of embedding two polydisks into a ball, as we take $d'\to\infty$, there must exist a column $\begin{pmatrix}e_{b1}^d\\\Lambda^1\\ \Lambda^2\end{pmatrix}$ where $d$ may be taken to be arbitrarily large.
The genus bound applied to this column becomes $$\sum_{i=1}^2 x(\Lambda^i)+y(\Lambda^i)\geq d(b+2)-1.$$

Also, the action inequality gives $$bcd=\mathcal{A}_{E(bc,c)}(e_{b1}^d)\geq \mathcal{A}_{P(1,1)}(\Lambda^1)+\mathcal{A}_{P(1,1)}(\Lambda^2)=\sum_{i=1}^2 x(\Lambda^i)+y(\Lambda^i).$$

Combining these inequalities and taking $d\to\infty$ gives the result.
\end{proof}

Finally, we consider the problem of symplectically embedding two polydisks into a square polydisk.

\begin{proof}[Proof of Theorem \ref{thm:two-polydisks-polydisk}]

We take $\Lambda'=e_{10}^{d} e_{01}^{d}$ with $d$ to be a large even integer, and choose the index split as $I_1=I_2=\frac{1}{2}I(\Lambda')=d^2+2d.$ As in the proof of Theorem \ref{thm:two-polydisks-ball}, we separate into cases.

Case 1: No column is repeated as $d\to\infty$.

Fix $\varepsilon>0$ and merge all columns of the form $$\begin{pmatrix}e_{10}^ae_{01}^b\\ \tilde{\Lambda^1}\\ \tilde{\Lambda^2}\end{pmatrix}$$ for which $\frac{1}{a+b}<\varepsilon.$ Let $\{C_k\}_{k\in S}$ be this set of columns. By assumption, there are only finitely many possible columns which violate this inequality, say $N=N(\varepsilon).$

Call the resulting merged column $$\begin{pmatrix}e_{10}^Ae_{01}^B\\ \Lambda^1 \\ \Lambda^2\end{pmatrix}.$$

There are at most $N$ columns that are excluded from this merge, all of the form $\begin{pmatrix}e_{10}^xe_{01}^y\\ \star \\ \star\end{pmatrix}$ with $x+y\leq 1/\varepsilon.$ Since the sum of the exponents of $e_{10}$ and $e_{01}$ in the first entry of all columns are both $d$, it follows that $d-A$ and $d-B$ are both bounded above by $N/\varepsilon.$ Also, for $k\in S$, if $C_k=\begin{pmatrix}e_{10}^{a_k}e_{01}^{b_k}\\ \Lambda^1_k \\ \Lambda^2_k\end{pmatrix},$ then

$$\frac{|S|}{\varepsilon}\leq\sum_{k\in S}(a_k+b_k)\leq\sum_{i=1}^n (a_i+b_i)=2d,$$

so we may bound $|S|\leq 2d\varepsilon.$

The action inequality then reads

\begin{align*}
(x(\Lambda^1)+x(\Lambda^2))+a(y(\Lambda^1)+y(\Lambda^2))&\leq c(A+B)\\ &\leq 2dc,
\end{align*} which we rearrange to obtain the constraint
$$c\geq\frac{1}{2d} \left(x(\Lambda^1)+x(\Lambda^2)+a(y(\Lambda^1)+y(\Lambda^2))\right).$$

The genus bound applied to the merged column reads:
\begin{align*}
x(\Lambda^1)+y(\Lambda^1)+x(\Lambda^2)+y(\Lambda^2)&\geq 2(A+B)-|S|\\ &\geq 4d-4N/\varepsilon-2d\varepsilon
\end{align*}

The index equality reads:

\begin{align*}
2(AB+A+B)&=I(\Lambda^1)+I(\Lambda^2)\leq 2(1+o(1))\sum_{i=1}^2 x(\Lambda^i)y(\Lambda^i)\\
\sum_{i=1}^2 x(\Lambda^i)y(\Lambda^i)&\geq\frac{1}{1+o(1)}(AB+A+B)\\ &\geq \frac{d^2+O(d)}{1+o(1)}
\end{align*}

Since $d$ may be taken to be arbitrarily large and $\varepsilon$ may be taken to be arbitrarily small, independently of $d$, we may take limits to find that

$$c\geq \min \left\{\frac{1}{2}\left(x_1+x_2+a(y_1+y_2)\right)\bigg\vert\ x_1y_1+x_2y_2\geq 1 \text{ and } x_1+y_1+x_2+y_2\geq 4\right\}.$$

Solving this optimization problem, we obtain:

\begin{align*}c\geq \begin{cases} 2+\left(1-\frac{\sqrt{3}}{2}\right)(a-1) & 1\leq a\leq 7+4\sqrt{3}\\\sqrt{a} & a\geq 7+4\sqrt{3}.\end{cases}
\end{align*}

\begin{remark} We could get a stronger bound in this case by estimating both indices $I(\Lambda^i)$ individually instead of their sum $I(\Lambda^1)+I(\Lambda^2)$ in the above optimization, in which case the bound for the large $a$ would agree with the volume constraint. However, this bound is already larger than what we obtain in Case 2, so it would not affect the statement of the theorem.
\end{remark}

Case 2: Some column is repeated when we take $d\to\infty.$ By Lemma \ref{lem:repeated-column}, any such column must be of the form $\begin{pmatrix}e_{10}^Re_{01}^S\\ \emptyset \\ e_{rs}\end{pmatrix}$ or $\begin{pmatrix}e_{10}^Re_{01}^S\\ e_{rs} \\ \emptyset\end{pmatrix}.$ We have sub-cases that either both forms of columns occur, or only one does.

Sub-case 1: Suppose that only one of these forms occur, without loss of generality, taking the second one. We may assume that $R$ and $S$ stay bounded as $d'\to\infty$, as otherwise the argument and bound for $c$ in Case 1 remain valid. For any such column, we have the relations $rs=2RS$ and $r+s=2(R+S)-1$ from Lemma \ref{lem:repeated-column}.

Now, let $\begin{pmatrix}e_{10}^Ae_{01}^B\\ \tilde{\Lambda^1}\\ \emptyset\end{pmatrix}$ be the merge of all columns with an empty generator, and let $C=\begin{pmatrix}e_{10}^{d-A}e_{01}^{d-B}\\ \Lambda^1\\ \Lambda^2\end{pmatrix}$ be the merge of the other columns. 

The action inequality applied to $C$ reads $$c(2d-A-B)\geq \sum_{i=1}^2 x(\Lambda^i)+ay(\Lambda^i).$$

By the assumed index split we have $I(\Lambda^2)=d^2$ to first order, and hence by similar estimates to the previous cases:

$$x(\Lambda^2)y(\Lambda^2)\geq \frac{1}{2}d^2+O(d).$$

The genus bound applied to $C$ gives

$$\sum_{i=1}^2 x(\Lambda^i)+y(\Lambda^i)\geq 2(2d-A-B)-k,$$

where $k$ is the number of columns merged to produce $C.$

Also, applying the index equality to the column $C$ produces.

$$x(\Lambda^1)y(\Lambda^1)+x(\Lambda^2)y(\Lambda^2)\geq (d-A)(d-B)+O(d).$$

We also have a constraint that if a column $\begin{pmatrix}e_{10}^Re_{01}^S\\ e_{rs} \\ \emptyset\end{pmatrix}$ occurs, then considering the index of the column obtained by merging it with $C$ gives

$$2c_{rs}(\Lambda_1)=I(e_{10}^{d-A+R}e_{01}^{d-B+S})-I(e_{10}^Re_{01}^S)-I(e_{10}^{d-A}e_{01}^{d-B})$$

from Lemma \ref{lem:index-shift}.

By the assumption that $R$ and $S$ stay bounded, we find $$c_{rs}(\Lambda_1)=R(d-B)+S(d-A)+O(1).$$ As in the analogous case in Theorem 1.2, this means that the point $(x(\Lambda^1),y(\Lambda^1))$ lies on or above the line $$sx+ry=R(d-B)+S(d-A)+O(1).$$ This gives a new constraint

\begin{align*}
\text{max}(r,s)(x(\Lambda^1)+y(\Lambda^1))&\geq sx(\Lambda^1)+ry(\Lambda^1)\\ &\geq S(d-A)+R(d-B)+O(1)\\ &\geq\text{max}(S,R)\cdot\text{min}(d-A,d-B)+O(1).
\end{align*}

The relation $r+s=2(R+S)-1$ implies that $\frac{\text{max}(S,R)}{\text{max}(s,r)}\geq\frac{1}{4},$ so

$$x(\Lambda^1)+y(\Lambda^1)\geq \frac{1}{4}\text{min}(d-A,d-B)+O(1).$$





Altogether, we may keep only leading order terms in $d$ and make $k/d$ arbitrarily small, by identical arguments to first case, to obtain

\begin{align*}
c &\geq \min\left\{\left(\frac{1}{2-u-v}\right)(X_1+aY_1+X_2+aY_2)\ \middle|\ 
\begin{array}{l}
\displaystyle\sum_{i=1}^2 X_i+Y_i\geq 2(2-u-v) \\[6pt]
X_2Y_2\geq \tfrac{1}{2} \\[6pt]
X_1Y_1+X_2Y_2\geq (1-u)(1-v) \\[6pt]
X_1+Y_1\geq \tfrac{1}{4}\min(1-u,1-v) \\[6pt]
0\leq u,v\leq 1 \\[6pt]
0\leq X_i,Y_i \\[6pt]
\end{array}\right\}\\
&=\begin{cases} 2+\left(\frac{23-\sqrt{401}}{24}\right)(a-1) & 1\leq a\leq a^*\\ \frac{2\sqrt{2a}}{3} & a\geq a^*,\end{cases}
\end{align*}

where $a^*=\frac{(23+\sqrt{401})^2}{128}\approx 14.46.$
In the above minimization problem, $X_i$ and $Y_i$ represent the quantities $x(\Lambda^i)/d$ and $y(\Lambda^i)/d$ and $u$ and $v$ represent $A/d$ and $B/d$, respectively. 


In this computation, our bound is stronger than the volume constraint for $1\leq a\leq 2.348.$


Sub-case 2 Now, suppose that there are repeated columns of both forms, i.e. there are columns $\begin{pmatrix}e_{10}^Re_{01}^S\\ \emptyset \\ e_{rs}\end{pmatrix}$ and $\begin{pmatrix}e_{10}^Pe_{01}^Q\\ e_{pq} \\ \emptyset\end{pmatrix}$ occuring in the factorization of Theorem 3.1. Also, let $C=\begin{pmatrix}e_{10}^{A}e_{01}^B\\ \Lambda^1\\ \Lambda^2\end{pmatrix}$ be a column appearing in the factorization that is not one of the two repeated columns, and where we may take $A+B$ to be arbitrarily large by taking $d\to\infty.$

By merging the first two columns with an empty generator and considering the index, we find that $$I(e_{10}^{R+P}e_{01}^{S+Q})=I(e_{10}^Re_{01}^S)+I(e_{01}^Re_{01}^Q),$$ which is only possible when either $R=P=0$ or $S=Q=0.$ In either case, the equalities of Lemma 3.3 guarantee that the nonzero exponent is $1.$ Now, if either $(r,s)$ or $(p,q)$ equals $(0,1)$, then the action inequality applied to either column reads $c\geq a,$ which is a stronger bound than our statement when $a\geq 2$.

The action inequality applied to $C$ then reads

\begin{align*}
c(A+B)&\geq (x(\Lambda^1)+x(\Lambda^2))+a(y(\Lambda^1)+y(\Lambda^2))\\&=\left(x(\Lambda^1)+x(\Lambda^2)+y(\Lambda^1)+y(\Lambda^2)\right)
\end{align*}

The righthand side is at least $2(A+B)-1$ by the genus bound, so by taking $A+B\to\infty$ we obtain $c\geq 2$, which is our claimed bound when $1\leq a\leq 2.$

Hence, if either $(r,s)$ or $(p,q)$ is $(0,1),$ then $c\geq\text{min}(a,2)$

So, in this sub-case, we are just left with the possibility $(r,s)=(p,q)=(1,0).$ Since the roles of $e_{10}$ and $e_{01}$ are symmetric for the positive boundary $P(c,c)$, we may assume without loss of generality that $(R,S)=(P,Q)=(1,0)$ and so the two repeated columns are $\begin{pmatrix}e_{10}\\ \emptyset \\ e_{10}\end{pmatrix}$ and $\begin{pmatrix}e_{10}\\ e_{10}\\ \emptyset\end{pmatrix}.$


Since both repeated columns have $e_{10}$ in the first component, we may assume that $B>A$ in the definition of the column $C.$

By merging the repeated columns with $C$ and comparing index, we find that $y(\Lambda^1)=y(\Lambda^2)=B.$

The action inequality applied to $C$ then reads

\begin{align*}
c(A+B)&\geq (x(\Lambda^1)+x(\Lambda^2))+a(y(\Lambda^1)+y(\Lambda^2))\\&=\left(x(\Lambda^1)+x(\Lambda^2)+y(\Lambda^1)+y(\Lambda^2)\right)+(a-1)\left(y(\Lambda^1)+y(\Lambda^2)\right)\\
&=\left(x(\Lambda^1)+x(\Lambda^2)+y(\Lambda^1)+y(\Lambda^2)\right)+2B(a-1)
\end{align*}

Combining with the genus bound $$x(\Lambda^1)+x(\Lambda^2)+y(\Lambda^1)+y(\Lambda^2)\geq 2(A+B)-1$$ and taking $A+B\to\infty$ gives the bound $c\geq 2+(a-1)=1+a$ (since $B>A$ by assumption).

So, altogether in the case that there are repeated columns of both forms, we find $c\geq\text{min}(2,a,1+a)=\text{min}(2,a),$ which agrees with our claimed bound for $c$ when $1\leq a\leq 2$ and is larger than it for $a>2.$

Taking the minimum of the possible lower bounds for $c$ as a function of $a$ in all above cases gives the statement of Theorem \ref{thm:two-polydisks-polydisk}.
\end{proof}

While we only calculated with convex toric domains in this section, we state a counterpart of Theorem \ref{thm:main-convex} for the concave case.

\begin{theorem}\label{thm:main-concave}
Let $X_{\Omega_i}$ for $i=1,\ldots,m$ and $X_{\Omega'}$ be concave toric domains. Suppose that there is a symplectic embedding $\varphi:\coprod_{i=1}^m X_{\Omega_i}\hookrightarrow X_{\Omega'}$. Let $\Lambda^1,\ldots,\Lambda^m$ be concave generators such that $\Lambda^i$ is maximal for $\Omega_i$ for each $i$. Then there exists a concave generator $\Lambda'$ and factorizations $\Lambda'=\Lambda_1'\ldots\Lambda_n'$ and $\Lambda^i=\Lambda_1^i\ldots\Lambda_n^i$ for each $i$ such that the following conditions hold:
\begin{itemize}
\item{(a)} We have $I(\Lambda')=\sum_{i=1}^m I(\Lambda^i).$
\item(b) For each subset $S\subseteq\{1,\ldots,n\}$, we have $\sum_{i=1}^m I(\prod_{j\in S}\Lambda_j^i)=I(\prod_{j\in S}\Lambda_j').$
\item(c) For each $j\in\{1,\ldots,n\}$, we have $\sum_{i=1}^m \mathcal{A}_{\Omega_i}(\Lambda_j^i)\leq \mathcal{A}_{\Omega'}(\Lambda_j').$
\item(d) For all $j\in\{1,\ldots,n\}$ we have $$\sum_{i=1}^m x(\Lambda_j^i)+y(\Lambda_j^i)-m(\Lambda_j^i)\geq x(\Lambda_j')+y(\Lambda_j')+\frac{1}{2}h(\Lambda_j')-1.$$
\item(e) Given $k,j\in\{1,\ldots,n\}$, if $\Lambda_k'\neq\Lambda_j'$ or if $\Lambda_k^i\neq\Lambda_j^i$ for any $i\in\{1,\ldots,m\}$, then $\Lambda'_k$ and $\Lambda'_j$ have no elliptic factors $e_{ab}$ in common.
\item(f) For all $j\in\{1,\ldots,n\}$ and all $i\in\{1,\ldots,m\}$, the following inequality holds whenever $I(\Lambda_j')>2$:
$$c_{11}(\Lambda_j')\geq c_{11}(\Lambda_j^i),$$
where $c_{ab}(\Lambda)$ for a concave generator $\Lambda$ is the value of $c$ for which the line $bx+ay=c$ is inwardly tangent to $\Lambda.$
\end{itemize}
\end{theorem}

\begin{remark}The main difference between this theorem and the version for convex toric domains is that the generators for the negative domains are chosen as inputs instead of a generator for the positive domain. Also, the genus bound in (d) is slightly different. 
\end{remark}

\begin{remark} This genus bound for the concave case was recently proven independently in Theorem 3.3 in \cite{Miranda} when $m=1.$
\end{remark}

\section{Review of Embedded Contact Homology for Convex and Concave Toric Domains}

In this section, we will review some notions from embedded contact homology and describe how they relate to convex and concave toric domains. The definitions and statements about ECH may all be found in \cite{LectureNotes} and the description of the ECH of convex and concave toric domains may be found in \cite{Beyond} and \cite{Concave}.
Given a closed contact $3$-manifold $Y$ with contact form $\lambda$ and contact structure $\xi=\ker(\lambda)$, define an \textit{ECH generator} for $(Y,\lambda)$ to be a (possibly empty) finite set of the form $\{(\alpha_i,m_i)\}$ where the $\alpha_i$ are distinct simple Reeb orbits on $Y$, modulo reparameterization, and $m_i\geq 1$ are integers such that $m_i=1$ whenever $\alpha_i$ is hyperbolic. The ECH chain group $ECC(Y,\lambda)$ is freely generated over $\Z/2$ as a vector space by such sets. Also, the \textit{action} of such an orbit set is defined as $\sum_i m_i\int_{\alpha_i}\lambda.$ This defines a filtration on each chain group: for each real number $L>0$, define $ECC^L(Y,\lambda)$ to be the subspace of $ECC(Y,\lambda)$ generated by orbit sets of action at most $L$.

Choose $J$ to be an almost complex structure on $\R\times Y$ that satisfies the following conditions. Below, $s$ denotes the $\R$-coordinate on $\R\times Y.$
\begin{itemize}
\item $J$ is translation-invariant in the sense that $\mathcal{L}_{\partial_s}J=0.$
\item On each cross section $\{s\}\times Y,$ identified with $Y$, we have that $J$ preserves the contact structure: $J\xi=\xi$, and rotates it positively in the sense that $d\lambda(v,Jv)>0$ \newline for $0\neq v\in\xi.$ 
\item $J\partial_s=R$ where $R$ is the Reeb vector field on $(Y,\lambda)$.
\end{itemize}

Given such a $J$, one may define a \textit{$J$-holomorphic curve from $\alpha$ to $\beta$} to be a (possibly disconnected) punctured Riemann surface $\Sigma$ with a $J$-holomorphic map $u:\Sigma\to\R\times Y.$ The punctures of $\Sigma$ are labeled as ``positive'' and ``negative'' so that near each positive puncture, $u$ is positively asymptotic to a cover of a simple Reeb orbit appearing in $\alpha$ and similarly for $\beta$ near the negative punctures. For each positive simple Reeb orbit $\alpha_i$, the sum of the covering multiplicities of each end asymptotic to a cover of $\alpha_i$ is to be the same as the multiplicity of $\alpha_i$ in $\alpha$, and similarly for the negative ends.
A \textit{$J$-holomorphic current} from $\alpha$ to $\beta$ is a formal sum $\mathcal{C}=\sum_k d_kC_k$ where the $C_k$ are distinct, connected, and somewhere-injective $J$-holomorphic curves, and the sum of the multiplicities of the $C_k$ at each positive or negative simple Reeb orbit equals the multiplicity of that Reeb orbit in $\alpha$ or $\beta$.
A $J$-holomorphic current $\mathcal{C}$ from $\alpha$ to $\beta$ represents a relative homology class $Z\in H_2(Y,\alpha,\beta).$ Given such a class, one defines its \textit{ECH Index} as

$$I(\alpha,\beta,Z)=c_\tau(Z)+Q_\tau(Z)+CZ_\tau^I(\alpha)-CZ_\tau^I(\beta),$$

where $c_\tau, Q_\tau$, and $CZ_\tau$ are, respectively, the relative first Chern number of $\xi$ over $Z$, the relative self-intersection number of $Z$, and the sum of Conley-Zehnder indices at each end of $Z.$ These quantities generally depend on $\tau$, which is a simultaneous choice of trivializations for $\xi$ treated as a $2$-plane bundle over each Reeb orbit appearing in $\alpha$ or $\beta$.
Also, define $\mathcal{M}_1(\alpha,\beta)$ to be the moduli space of $J$-holomorphic currents from $\alpha$ to $\beta$ that have ECH index of $1.$ Since $J$ is assumed to be $\R$-invariant, there is a natural $\R$-action on this space by translation.

The differential $\partial:ECC(Y,\lambda)\to ECC(Y,\lambda)$ is now defined as

$$\partial \alpha=\sum_{\beta} \#_2\left(\frac{\mathcal{M}_1(\alpha,\beta)}{\R}\right) \beta,$$

for generic $J$, where $\#_2$ denotes the mod 2 cardinality of a finite set and the sum is over all ECH generators $\beta$. This differential squares to zero and we define $ECH(Y,\lambda)$ to be the resulting homology. This does not depend on the choice of $J$ or the contact form representing the contact structure $\xi=\ker(\lambda)$, up to isomorphism. When $\lambda$ is understood, we will write it as just $ECH(Y)$. Also, an application of Stokes' Theorem to curves in the above moduli space shows that the differential may not increase the action of orbit sets; i.e., $ECC^L(Y,\lambda)$ is a subcomplex of $ECC(Y,\lambda)$ for any $L>0$, and hence, we obtain filtered homology groups $ECH^L(Y,\lambda)$ as well.
When $Y$ is diffeomorphic to $S^3$ (or, more generally, has $H_2(Y,\Z)=0$), the ECH index can also be used to define a $\Z$-grading on $ECC(Y,\lambda)$ via

$$I(\alpha)=I(\alpha,\emptyset,Z),$$

where $Z$ is the unique class of $H_2(Y,\alpha,\emptyset).$ (For general $Y$, this number will depend on the choice of $Z$ and only define a relative grading). With respect to this grading, the differential $\partial$ has degree $-1$. We may also define the quantities $c_\tau(\alpha)$ and $Q_\tau(\alpha)$ analogously. 
Suppose that $(Y,\lambda)$ is connected. Then there is a $\textit{$U$-map}$ of degree $-2$:
$$U:ECH_*(Y,\lambda)\to ECH_{*-2}(Y,\lambda),$$

defined similarly to the differential, except that instead of counting index $1$ curves modulo translation, the $U$-map on the chain level counts index $2$ curves passing through a fixed base-point. On homology, this induces a map that does not depend on the choice of generic basepoint. When $Y=S^3$ with the standard contact structure, it is calculated in \cite{LectureNotes} that $ECH(S^3)$ has a unique nonzero element of index $2k$ for $k\geq 0,$ which we will denote by $\zeta_k$. These elements form a \textit{$U$-chain} in the sense that $U\zeta_k=\zeta_{k-1}$ for $k\geq 1$. Since the boundary of a convex or concave toric domain is  contactomorphic to the standard $S^3$ via the map $y\mapsto \frac{y}{|y|}$, the same description is valid for them.

When $Y$ is disconnected, the $U$-map on homology will depend on which component of $Y$ the chosen basepoint lies in, and there is accordingly a resulting $U$-map for each component of $Y$.

We now describe the chain group $ECC(Y)$ when $Y$ is the boundary of a convex or concave toric domain.

Let $Y=\partial X_{\Omega}$ be a convex or concave toric domain with smooth defining function $f$ and let $(A,0)$ be the $x$-intercept of its defining curve. Then $\mu^{-1}(t,f(t))$ is always a torus when $t$ and $f(t)$ are both nonzero. The Reeb vector field is tangent to this torus, and when $f'(t)=-b/a$ is rational in reduced form, this torus is foliated by Reeb orbits of action $af(t)+bt.$ Call this torus $T_{ab}.$ The Reeb orbits of $Y$ are those on the tori $T_{ab}$ together with $e_{01}:=\mu^{-1}(0,f(0))$ and $e_{10}:=\mu^{-1}(A,0).$ Both $e_{10}$ and $e_{01}$ are elliptic.
We may perturb the contact form $\lambda_Y$ by multiplying it by a smooth function $f$ that is $C^0$-close to the constant function $1$. For any $L>0$ and any $\varepsilon>0$, this function may be chosen that $||f-1||_{C^0}<\varepsilon$ and after making this perturbation, each torus $T_{ab}$ of Reeb orbits with action smaller than $L$ is replaced by a pair of Reeb orbits $e_{ab}$ and $h_{ab},$ where $e_{ab}$ is elliptic and has slightly larger action than the orbits of the torus it corresponds to, and $h_{ab}$ is hyperbolic with slightly smaller action when $X_{\Omega}$ is convex, and vice-versa when concave. No other Reeb orbits of action smaller than $L$ are created. A trivialization for $\xi$ restricted to any $e_{ab}$ may be chosen so that the rotation angle is arbitrarily close to zero by choosing the perturbation $f$ appropriately, slightly positive in the convex case, and slightly negative in the concave case. Similarly, above every hyperbolic orbit $h_{ab}$, the contact structure $\xi$ may be trivialized so $h_{ab}$ has a Conley-Zehnder index of zero.  
It follows that ECH generators for convex or concave toric domains generated by Reeb orbits with action smaller than a fixed bound $L$ and don't contain $e_{10}$ or $e_{01}$ in the concave case may be identified with convex or concave generators as in Definitions \ref{def:convex-generator} and \ref{def:concave-generator}. We will see in the next section why we exclude $e_{10}$ and $e_{01}$ in the concave case.
For an ECH generator $\alpha$ associated with a convex or concave generator $\Lambda$, we may calculate the quantities $c_\tau(\alpha),Q_\tau(\alpha),$ and $I(\alpha)$ in terms of $\Lambda$ as follows:

\begin{itemize}
\item $c_\tau(\alpha)=x(\Lambda)+y(\Lambda)$
\item $Q_\tau(\alpha)=2\text{Area}(\Lambda)$ 
\item $I(\alpha)=I(\Lambda),$ 
\end{itemize}

where $I(\Lambda)$ is as defined in Section 2 for convex and concave generators.

Lastly, we give some background about cobordism maps in ECH. 
Given two closed (not necessarily connected) 3-dimensional contact manifolds $(Y^+,\lambda^+)$ and $(Y^-,\lambda^-),$ a \textit{weakly exact cobordism} from $(Y^+,\lambda^+)$ to $(Y^-,\lambda^-)$ is a 4-dimensional symplectic manifold $(X,\omega)$ such that
\begin{itemize}
\item $\partial X=Y^+-Y^-,$ where the orientation of $X$ is determined by $\omega\wedge\omega$, and the orientations of $Y^\pm$ are determined by $\lambda^\pm\wedge d\lambda^\pm$.
\item The restriction of $\omega$ to $Y^{\pm}$ equals $d\lambda^{\pm}.$
\end{itemize}

If, in addition, the symplectic form $\omega$ is exact, then we say that $(X,\omega)$ is a \textit{weakly exact} cobordism.

Given a weakly exact cobordism as above, we have cobordism maps on (filtered) ECH for any $L>0$:

$$\Phi^{X,L}:ECH^L(Y^+,\lambda^+)\to ECH^L(Y^-,\lambda^-)$$

satisfying the properties

\begin{itemize}
\item $\Phi^{X,L}([\emptyset])=[\emptyset]$ and
\item $\Phi^{X,L}\circ U_+=U_-\circ \Phi^{X,L},$
\end{itemize}

where $U_\pm$ are $U$-maps on $ECH^L(Y^\pm,\lambda^\pm)$ coming from components of $Y^\pm$ lying in the same component of $X.$

That there are cobordism maps satisfying these properties is Theorem 2.3 in \cite{Quant}. We will summarize the second property by saying that ``cobordism maps commute with $U$-maps''.

\section{Proofs of Theorems 3.1 and 3.11}
In this section, we will prove our main theorem and sketch a proof of a similar theorem for concave toric domains.

\begin{proof}[Proof of Theorem \ref{thm:main-convex}]

Put $k=\frac{1}{2}I(\Lambda').$ For $Y$ the boundary of a convex or concave toric domain and $p$ a nonnegative even integer, let $\sigma^Y_p$ be the nonzero class in $ECH_{p}(Y).$
Since $\Lambda'$ is minimal, the orbit set $\alpha$ it corresponds to represents a nonzero ECH class by Lemma 5.5 in \cite{Beyond}, which must be $[\alpha]=\sigma^{X_{\Omega'}}_{2k}.$ 
Next, note that $X=X_{\Omega'}\setminus\left(\coprod_{i=1}^m\varphi(X_{\Omega_i})\right)^\circ$ is a cobordism from $\partial X_{\Omega'}$ to $\coprod_{i=1}^m\partial X_{\Omega_i}.$ Since $\varphi$ is a symplectic embedding, the restriction of the standard symplectic form on $X_{\Omega'}$ to $X$ gives $X$ the structure of a weakly exact cobordism. Fix $L$ to be a real number larger than the action of $\alpha$ and consider the associated cobordism map $\Phi^{X,L}:\text{ECH}^L(\partial X_{\Omega'})\to\text{ECH}^L(\coprod_{i=1}^m \partial X_{\Omega_i}).$ We will suppress writing the action bound $L$ in the rest of the proof, except when it is needed.




Under the canonical isomorphism 

\begin{equation}\label{eq:tensor-decomp}
ECH(\coprod_{i=1}^m \partial X_{\Omega_i})\cong\bigotimes_{i=1}^m ECH(\partial X_{\Omega_i}),\end{equation}

we may expand

\begin{equation}\label{eq:phi-expansion}\Phi^X([\alpha])=\sum_{I:i_1+\ldots+i_m=2k} c_I \ \sigma_{i_1}^{X_{\Omega_1}}\otimes\ldots\otimes\sigma_{i_m}^{X_{\Omega_m}},\end{equation}



where the sum is over all partitions $I$ of $2k$ into an ordered sum of $m$ nonnegative even integers. We have a $U$ map on the lefthand side of (\ref{eq:tensor-decomp}) for each of the components $X_{\Omega_i}.$ Under this isomorphism, the $j^{\text{th}}$ $U$-map on the righthand side is given by $$\text{Id}_1\otimes\ldots\otimes U_j\otimes\ldots\otimes\text{Id}_m$$ for any $j\in\{1,\ldots,m\}$ where $\text{Id}_i$ is the identity map on $ECH(\partial X_{\Omega_i})$ and $U_j$ is the $U$-map on $ECH(\partial X_{\Omega_j}).$ We consider the class $\Phi^X(U^k[\alpha]).$ On one hand, $U^k[\alpha]=[\emptyset]$ from our description of the ECH of $S^3$, so $$\Phi^X(U^k[\alpha])=\Phi^X([\emptyset])=[\emptyset]\neq 0.$$ On the other hand, since cobordism maps commute with $U$-maps, $\Phi(U^k[\alpha])$ may be calculated by applying $U_1^{j_1}\otimes\ldots\otimes U_m^{j_m}$ to the righthand side of (\ref{eq:phi-expansion}) for any partition \newline $j_1+\ldots+j_m=k.$  Choosing the partition $j_l=\frac{1}{2}I_l,$ we see that all terms in the sum vanish by index considerations except for one corresponding to the partition $I:i_l=I_l,$ so $c_I\neq 0.$
Next, let $$\phi^X:ECC(\partial X_{\Omega'})\to ECC(\coprod_{i=1}^m \partial X_{\Omega_i})$$ be a cobordism map on the chain level (still filtered by action up to $L$) that induces $\Phi^X$ and counts broken $J$-holomorphic currents, in the sense that if an orbit set appears as a summand in $\phi^X(\alpha)$, then there is a broken $J$-holomorphic current in $\overline{X}$ from $\alpha$ to that orbit set, for some suitably chosen almost complex structure $J$ on the completed cobordism $\overline{X}$ (see Section 6.2 in \cite{Calabi} for the definition of when $J$ is admissible for $\overline{X}$, and Proposition 6.2 in \cite{Calabi} for a description of these chain-level cobordism maps in terms of $J$). By the previous analysis, $\phi^X(\alpha)$ includes a term as a summand $\beta^1\otimes\ldots\otimes\beta^m$ contributing to a representative for $\sigma^{X_{\Omega_1}}_{I_1}\otimes\ldots\otimes\sigma^{X_{\Omega_m}}_{I_m}$ where $0\neq \beta^i\in ECC_{I_i}(\partial X_{\Omega_i})$ are ECH generators. Let $\Lambda^1,\ldots,\Lambda^m$ be the convex generators corresponding to these ECH generators. 

Let $\mathcal{C}$ be an index zero holomorphic current counted by the coefficient $$\langle\phi^X(\alpha),\beta^1\otimes\ldots\otimes\beta^m\rangle\neq 0.$$ 

We write the current $\mathcal{C}$ as a sum of components $\mathcal{C}=\sum_{j=1}^n C_j.$ Compared to the definition of a current given in Section 4, we do not include coefficients for the component curves $C_j$, but instead allow a given curve to appear with multiplicity in the sum. We define the factorizations of the convex generators by setting $\Lambda_j^i$ to correspond to  the ends of $C_j$ in $\partial X_{\Omega_i},$ and similarly $\Lambda'_j$ for the ends of $C_j$ in $\partial X_{\Omega'}.$ Also, let $\beta_j^i$ and $\beta_j'$ respectively be the orbit sets corresponding to $\Lambda_j^i$ and $\Lambda'_j.$
 
The cobordism $(X,\omega,J)$ may be taken to be $L$-tame in the sense of Section 4.1 in \cite{Beyond} by an identical argument as given in the proof of Theorem 1.19 there. Hence each $C_i$ appearing in $\mathcal{C}$ is an unbroken holomorphic curve and each sub-current of $\mathcal{C}$ (in particular, each individual component $C_j$) also has index zero by the same argument as given in the proof of Theorem 4.8 in \cite{Beyond}.

When $C_i$ and $C_k$ are distinct curves appearing as components in $\mathcal{C}$, they may not have any elliptic ends in common, as this would violate the equality condition for the union bound in $0=I(C_1)+I(C_2)\leq I(C_1\cup C_2)=0.$ The argument is the same as in Proposition 4.8 in \cite{LectureNotes}. This proves (f).

It is easily seen from the definition of ECH index that when $\beta'$ and $\beta''$ are orbit sets in different components of a given boundary, we have $I(\beta'\beta'')=I(\beta')+I(\beta''),$ where multiplication of generators is interpreted as adding the multiplicities of each involved Reeb orbit. For $i\neq k$, since $\beta_j^i$ and $\beta_j^k$ are orbit sets in distinct boundaries $\partial X_{\Omega_i}$ and $\partial X_{\Omega_k}$ (and likewise for the positive orbit sets), it follows that 

$$0=I(C_j)=I(\beta_j')-\sum_{i=1}^m I(\beta_j^i)=I(\Lambda'_j)-\sum_{i=1}^m I(\Lambda_j^i),$$

proving (b). Since each sub-current of $\mathcal{C}$ also has index zero, item (c) follows similarly.

By a standard Stokes' Theorem argument applied to the curve $C_j$, we have \newline$\mathcal{A}_{X_{\Omega'}}(\beta_j^i)\geq\sum_{i=1}^m\mathcal{A}_{X_{\Omega_i}}(\beta_j^i).$ Since the action of a Reeb orbit $e_{ab}$ or $h_{ab}$ on a perturbed boundary may be made arbitrarily close to the action of an orbit on the torus $T_{ab}$ for the unperturbed boundary by taking the perturbation to the contact form to be small enough, and $\mathcal{A}_{\Omega}(\Lambda)$ equals the action of the orbit set corresponding to $\Lambda$ on the (unperturbed) boundary of $X_\Omega$, the action inequality (d) follows.

The inequality in part $(e)$ may be given as in \cite{Beyond} as a result of facts about the topological complexity $J_0$ of a curve. We instead give a direct proof from the relative adjunction formula, with the same types of arguments as we will see in the appendix. In the following, $n(C_j)$ is the number of ends (positive or negative) of the curve $C_j,$ and $n^\pm(C_j)$ are the number of positive and negative ends. Also, $n_e^\pm$ and $n_h^\pm$ are respectively the number of elliptic and hyperbolic orbits at positive/negative ends of $C_j$.  

We recall the relative adjunction formula (see \cite{LectureNotes} Section 3.3 in the case of a symplectization, though it remains unchanged for the cobordism case)

$$c_\tau(C)=\chi(C)+Q_\tau(C)+w_\tau(C)-2\delta(C),$$

where $C\in\mathcal{M}(\alpha,\beta)$ is a somewhere injective curve and $\delta(C)$ is a weighted count of singularities of $C.$ For generic $J$, connected curves $C$ of ECH index zero in a cobordism have $\delta(C)=0$ by the index inequality $\text{ind}(C)\leq I(C)-2\delta(C)$ (see 3.3 in \cite{LectureNotes} for the statement in symplectizations and Section 3.9 for the discussion of the cobordism case). As $J$ is generic, the Fredholm index on the lefthand side must be non-negative, which forces $\delta(C)=0$ too.  We apply this to $C=C_j$ in our factorization and use the formulae in Section 4 to evaluate the terms

\begin{itemize}
\item $c_\tau(C_j)=x(\Lambda'_j)+y(\Lambda'_j)-\sum_{i=1}^m x(\Lambda_j^i)+y(\Lambda_j^i)$
\item $Q_\tau(C_j)=2\text{Area}(\Lambda'_j)-2\sum_{i=1}^m \text{Area}(\Lambda_j^i)$ 
\end{itemize}

and hence obtain

$$
x(\Lambda_j')+y(\Lambda_j')-\sum_{i=1}^m x(\Lambda_j^i)+y(\Lambda_j^i)=2\text{Area}(\Lambda_j')-2\sum_{i=1}^m\text{Area}(\Lambda_j^i)+w_\tau(C_j)+\chi(C_j).
$$

To simplify this, we use the identity $$2\text{Area}(\Lambda)=I(\Lambda)-(x(\Lambda)+y(\Lambda)+m(\Lambda)-h(\Lambda)),$$ which follows from Pick's Theorem. Note that $h(\Lambda_j')=0$ by assumption.

\begin{align*}
2\left(x(\Lambda_j')+y(\Lambda_j')\right)-2\sum_{i=1}^m x(\Lambda_j^i)+y(\Lambda_j^i)&=\sum_{i=1}^m \left(m(\Lambda_j^i)-h(\Lambda_j^i)\right)-m(\Lambda_j')\\&+w_\tau(C_j)+2-2g(C_j)-n(C_j)
\end{align*}

We use the trivial bound $g(C_j)\geq 0$ and rearrange to obtain

\begin{align*}
\sum_{i=1}^m x(\Lambda_j^i)+y(\Lambda_j^i)-\frac{1}{2}h(\Lambda_j^i)&\geq x(\Lambda_j')+y(\Lambda_j')-\frac{1}{2}\sum_{i=1}^m m(\Lambda_j^i)+\frac{1}{2}m(\Lambda_j')-1\\ &+\frac{1}{2}\left(n-w_\tau(C)\right).
\end{align*}

We evaluate the last term as follows. Since $0=\text{ind}(C_j)=I(C_j)=0$, the equality condition for the index inequality in \cite{LectureNotes} tells us each positive end of $C_j$ is asymptotic to a simple Reeb orbit and no two ends of $C_j$ are asymptotic to the same orbit. We may write $w_\tau(C_j)=w_\tau^+(C_j)-w_\tau^-(C_j)$ where $w_\tau^\pm(C_j)$ is the contribution to $w_\tau(C_j)$ from the positive/negative orbits.  From the positive partition conditions, we have $w_\tau^+(C_j)=CZ^{I,+}(C_j)-CZ^{\text{ind},+}(C_j),$ where the superscript $+$ denotes the contribution to the Conley-Zehnder term from only the positive ends. In our trivialization, each end at an elliptic orbit contributes $1$ to the $CZ^{\text{ind},+}$ term and contributes its multiplicity to $CZ^{I,+}$, which is also $1.$ So, altogether, $w_\tau^+(C_j)=0$


By the above negative partition conditions, $w_\tau^-(C_j)= CZ^I(C_j)-CZ^{\text{ind}}(C_j)$ from the writhe bound. In our trivialization, each elliptic orbit $e_{ab}$ has a Conley-Zehnder index of $1$ and each hyperbolic orbit $h_{ab}$ has Conley-Zehnder index of $0$. Since each hyperbolic orbit has multiplicity $1$, we get a total contribution of $\sum_{i=1}^m h(\Lambda_j^i)-\sum_{i=1}^m m(\Lambda_j^i)$ to $CZ^I(C_j)$ from the negative orbits. We similarly see that the contribution to $CZ^\text{ind}(C)$ from the negative orbits may be computed as 

\begin{align*}n_e^-(C_j)&=n^-(C_j)-n_h^-(C_j)\\ &=n^-(C_j)-\sum_{i=1}^m h(\Lambda_j^i).
\end{align*}Substituting this data into the above yields

\begin{align*}
\sum_{i=1}^m x(\Lambda_j^i)+y(\Lambda_j^i)-\frac{1}{2}h(\Lambda_j^i)&\geq x(\Lambda_j')+y(\Lambda_j')+\frac{1}{2}m(\Lambda_j')-1-\frac{1}{2}\sum_{i=1}^m m(\Lambda_j^i)+\frac{1}{2}n(C_j)\\ &-\frac{1}{2}\left(-\sum_{i=1}^m m(\Lambda_j^i)+\sum_{i=1}^m h(\Lambda_j^i)+n^-(C)-\sum_{i=1}^m h(\Lambda_j^i)\right),
\end{align*}
which we may rearrange to achieve the genus bound: 
\begin{align*}
\sum_{i=1}^m x(\Lambda_j^i)+y(\Lambda_j^i)-\frac{1}{2}h(\Lambda_j^i)&\geq x(\Lambda_j')+y(\Lambda_j')+\frac{1}{2}m(\Lambda_j')-1+\frac{1}{2}n^+(C_j)\\
&= x(\Lambda_j')+y(\Lambda_j')+m(\Lambda_j')-1.
\end{align*}

Finally, condition (g) comes from intersection positivity. It is proved in \cite{Quant} that $I(C_1+C_2)\geq I(C_1)+I(C_2)+2C_1\cdot C_2,$ where $C_1\cdot C_2$ is the the algebraic count of intersections of $C_1$ with $C_2.$ In particular, $I(C_1+C_2)\geq I(C_1)+I(C_2)$ when $C_1$ and $C_2$ have no components in common.

Let $C_1$ be the curve $C_j$ in the factorization described above. Let $C_2$ be a $J$-holomorphic curve guaranteed by proof of the main theorem with positive end given by $e_{10}$ with index split $I_j=2$ (and hence $I_k=0$ for $k\neq j$). Both curves are necessarily connected and the index assumption $I(\Lambda_j')>2$ guarantees that they are not equal. The only two cases for the negative ends of $C_2$ are $e_{10}$ and $e_{01}.$ 

If the negative end of $C_2$ is given by $e_{10},$ then the above index inequality becomes $$I(\Lambda'_j e_{10})-I(\Lambda^i_j e_{10})\geq I(\Lambda'_j)-I(\Lambda^i_j).$$ Counting lattice points shows that $I(\Lambda e_{10})=I(\Lambda)+2y(\Lambda)+2$ for any convex generator $\Lambda,$ and substituting this gives $y(\Lambda'_j)\geq y(\Lambda^i_j).$ If instead, the negative end of $C_2$ is given by $e_{01}$, then we get an inequality $y(\Lambda'_j)\geq x(\Lambda^i_j),$ so in either case $y(\Lambda'_j)\geq\text{min}(x(\Lambda_j^i),y(\Lambda_j^i)).$ Repeating the argument with a curve $C_2$ with a positive end of $e_{01}$ completes the argument.
\end{proof}

Most of Theorem \ref{thm:main-concave} may be proven completely analogously to Theorem \ref{thm:main-convex}, so we just explain the differences here.
\begin{proof}

Given a concave toric domain $X_{\Omega}$ and a positive integer $M,$ we may slightly perturb the defining function $f$ so that $f'(0)<-M$ and $-1/M<f'(a)<0$ where $f(a)=0.$ After this perturbation, it is calculated in \cite{Concave} that the index of any ECH generator for $\partial X_{\Omega}$ that contains an $e_{10}$ or $e_{01}$ factor is greater than or equal to $2M$. Hence, by taking $M$ sufficiently large in what follows, we may assume that no ECH generators that appear have any $e_{10}$ or $e_{01}$ factors and thus the generators we consider are in correspondence with the concave generators of Definition \ref{def:concave-generator}.

Furthermore, in \cite{Concave}, it is shown in slightly different language that if $X_{\Omega}$ is a concave toric domain, then its ECH capacities are given by 
$$c_k(X_{\Omega})=\text{max}\left\{\mathcal{A}_{\Omega}(\Lambda):I(\Lambda)=2k \text{ and all edges of } \Lambda \text{ are labeled } ``e"\right\},$$
where $\Lambda$ is a concave generator.


It follows from the definition of ECH capacities that when $\Lambda$ is maximal for $\Omega$ with index $I(\Lambda)=2k$, every $\alpha\in ECC(\partial X_{\Omega})$ that represents the nonzero element of $ECH_{2k}(\partial X_{\Omega})$ must contain the orbit set corresponding to $\Lambda$ as a summand.

Now, let $\Lambda^1,\ldots,\Lambda^m$ be maximal for $\Omega_1,\ldots,\Omega_m$ and let $I=\sum_{i=1}^m I(\Lambda^i).$ 

As in the proof of Theorem \ref{thm:main-convex}, let $$\phi:ECC(\partial X_{\Omega'})\to \bigotimes_{i=1}^m ECC(\partial X_{\Omega_i})$$ be a chain map which induces the cobordism map on ECH and counts $J$-curves. Let \newline$\alpha\in ECC(\partial X_{\Omega'})$ represent the nonzero class of index $I.$ By the above reasoning, $\phi(\alpha)$ contains $\alpha_1\otimes\ldots\otimes\alpha_m$ as a summand.

When writing $\alpha=\sum_{i=1}^k\alpha_i'$ as a sum of ECH generators of index $I$, it follows that there is a $J$-holomorphic current $\mathcal{C}$ from some $\alpha'_i$ to the orbit set corresponding to $\alpha_1\otimes\ldots\otimes\alpha_m.$ Taking $\Lambda'$ to be the concave generator corresponding to $\alpha'_i,$ the factorizations of $\Lambda'$ and $\Lambda^1,\ldots,\Lambda^m$ correspond to decomposing $\mathcal{C}$ as a weighted sum of $J$-holomorphic curves, as in the proof of Theorem \ref{thm:main-convex}.

The genus bound is similar to the convex case, except that $\tau$ is chosen so that the Conley-Zehnder index of an elliptic orbit $e_{ab}^k$ is $-1$ instead of $1,$ and the index formula is slightly different. Given a concave generator $\Lambda$, applying Pick's Theorem to the index formula for a concave generator yields $2\text{Area}(\Lambda)=I(\Lambda)-x(\Lambda)-y(\Lambda)+m(\Lambda)-h(\Lambda).$ From considering the relative adjunction formula applied to the $j^{\text{th}}$ column of the matrix as in the convex case, we see

\begin{align*}
x(\Lambda_j')+y(\Lambda_j')-\sum_{i=1}^m \left(x(\Lambda_j^i)+y(\Lambda_j^i)\right)&=2\text{Area}(\Lambda_j')-2\sum_{i=1}^m\text{Area}(\Lambda_j^i)+2-2g-n+w_\tau(C_j).
\end{align*}

Applying Pick's formula in the same manner as in the convex case and rearranging yields:

\begin{align*}
2\left(x(\Lambda_j')+y(\Lambda_j')\right)-2\sum_{i=1}^m \left(x(\Lambda_j^i)+y(\Lambda_j^i)\right)&=\left(m(\Lambda_j')-\sum_{i=1}^m m(\Lambda_j^i)\right)-h(\Lambda_j')+2-2g-n+w_\tau(C_j).
\end{align*}

Finally, we use $g(C_j)\geq 0$ to obtain

\begin{align*}
\sum_{i=1}^m x(\Lambda_j^i)+y(\Lambda_j^i)&\geq x(\Lambda_j')+y(\Lambda_j')-1+\frac{1}{2}\left(\sum_{i=1}^m m(\Lambda_j^i)-m(\Lambda_j') \right)\\ &+\frac{1}{2}h(\Lambda_j')+\frac{1}{2}(n-w_\tau(C_j)).
\end{align*}

As in the convex case, since $\text{ind}(C_j)=I(C_j),$ we find that $w_\tau(C_j)=CZ^I(C_j)-CZ^{\text{ind}}(C_j)$ from the equality condition for the index inequality. Compared to the convex case, the partition conditions are reversed (see Section 3.9 in \cite{LectureNotes}). So, each positive simple elliptic orbit (at which $C_j$ has an end at a cover of) contributes $-1$ to $CZ^{\text{ind}}(C_j)$ and the negative of its multiplicity to $CZ^I(C_j)$. Each positive hyperbolic end contributes $0$ to both, so, altogether, we have $$w_\tau^+(C_j)=n_e^+(C_j)-m_e(\Lambda_j')=n^+(C_j)-m(\Lambda_j'),$$ where $m_e$ denotes the total multiplicity of all elliptic factors. The second equality holds because each positive hyperbolic orbit contributes $1$ to both terms on the righthand side. Each negative end (at an elliptic orbit by assumption) contributes $+1$ to both $CZ^{I}$ and to $CZ^{\text{ind}}$ since these ends are simple. So, altogether, $w_\tau(C_j)=w_\tau^+(C_j)=n^+(C_j)-m(\Lambda_j')$ with the same notation as in the proof of Theorem \ref{thm:main-convex}.

Hence, we find 
\begin{align*}
\frac{1}{2}(n(C)-w_\tau(C_j))&=\frac{1}{2}\left(n^+(C_j)+n^-(C_j)-n^+(C_j)+m(\Lambda'_j)\right)\\ &=\frac{1}{2}n^{-}(C_j)+\frac{1}{2}m(\Lambda_j')\\
&=\frac{1}{2}\sum_{i=1}^m m(\Lambda_j^i)+\frac{1}{2}m(\Lambda_j'),
\end{align*}
and the genus bound follows.
\end{proof}

Item (e) also follows from the partition conditions being reversed from the convex case.

The inequalities $c_{11}(\Lambda_j')\geq c_{11}(\Lambda_j^i)$ follow from intersection positivity and Lemma \ref{lem:index-shift} as in the convex case, except that now the only index $2$ generator is $e_{11}$, instead of $e_{10}$ or $e_{01}$.

\section{Isotopies of Symplectic Embeddings}

In this section, we will use the theory developed in previous sections to give constraints for when two symplectic embeddings of toric domains may be isotopic to each other.

For this, we need a somewhat stronger notion of cobordism.

We recall from \cite{LectureNotes} that a weakly exact symplectic cobordism $(X,\omega)$ from $(Y^+,\lambda^+)$ to $(Y^-,\lambda^-)$ is \textit{strongly exact} if there exists a $1$-form $\lambda$ on $X$ such that $d\lambda=\omega$ and $\lambda$ restricted to $Y^\pm$ equals $\lambda^\pm$. This is stronger than the condition $\omega\vert_{Y^\pm}=d\lambda\vert_{Y^\pm}$ that we needed previously.

Fortunately, this condition is still met in the cases we consider.

\begin{lemma}
Let $(X,\omega)$ be a weakly exact cobordism from $(Y^+,\lambda^+)$ to $(Y^-,\lambda^-).$ Suppose that $H^1(Y^\pm,\mathbb{R})=0.$ Then $X$ is a strongly exact cobordism.
\end{lemma}

\begin{proof}
Let $\lambda$ be a primitive for $\omega.$ Then $\left(\lambda\vert_{Y^-}\right)-\left(\lambda^-\right)$ is closed, and hence equal to $df^-$ for some smooth function $f^-:Y^-\to\R.$ Choose $f^+:Y^+\to\R$ similarly. Let $f:X\to\R$ be a smooth function that agrees with $f^\pm$ on $Y^\pm.$ Then $\lambda-df$ is a primitive for $\omega$ that agrees with $\lambda^\pm$ on $Y^\pm.$
\end{proof}

In particular, this holds when $Y^\pm$ is the boundary of a toric domain.
The importance of a cobordism being strongly exact is the following. 

\begin{proposition}
Suppose that $f_0,f_1:X^-\hookrightarrow X^+$ are symplectic embeddings and the associated cobordisms are strongly exact. Suppose also that they are isotopic by a path of symplectic embeddings in the sense that there is a smooth map $F:[0,1]\times X^-\to X^+$ such that $F(i,\cdot)=f_i$ for $i=0,1$ and for each $t\in [0,1],$ the map $f_t=F(t,\cdot):X^-\to X^+$ is a symplectic embedding representing a strongly exact cobordism. Then, viewing each $f_i$ as a cobordism from $\partial X^+$ to $\partial X^-$, any two associated cobordism maps $\phi_i:ECC_*(\partial X^+)\to ECC_*(\partial X^-)$ are chain homotopic, i.e. there exists a $\Z/2$-linear map $K:ECC_*(\partial X^+)\to ECC_{*+1}(\partial X^-)$ such that

$$\partial^- K+K\partial^+=\phi_1-\phi_0.$$

Here $\partial^\pm$ is the differential on $ECC(\partial X^\pm).$ Furthermore, the chain homotopy $K$ is action preserving: for all ECH generators $\alpha\in ECC_*(\partial X^+)$, if $\beta$ is a generator for $ECC_{*+1}(\partial X^-)$ such that $\langle K(\alpha),\beta\rangle\neq 0,$ then $\mathcal{A}(\alpha)\geq\mathcal{A}(\beta).$
\end{proposition}

This is proved from the the Homotopy Invariance Property in Prop 6.2 in \cite{Calabi}. We may view the $1$-parameter family of cobordisms from the isotopy $f_t$ as a single cobordism with varying contact forms $\lambda_t^\pm.$ We may take a family $\lambda_t$ of $1$-forms on the cobordism restricting to $\lambda_t^\pm$ on $\partial X^{\pm}$ to be smooth to produce the required strong homotopy. 
We also need to explicitly describe a cobordism map in a simple case.

\begin{lemma}[Characterization of cobordism map for low index generators]
Let $X_{\Omega^-}\subset X_{\Omega^+}$ be an inclusion of convex toric domains. Then, for suitably chosen admissible $J,$ a cobordism map $\phi:ECC_*(\partial X_{\Omega^+})\to ECC_*(\partial X_{\Omega^-})$ associated to the inclusion satisfies $$\phi(e_{10}^+)=e_{10}^- \text{ and } \phi(e_{01}^+)=e_{01}^-,$$

where $e_{10}^\pm$ and $e_{01}^\pm$ are the $e_{10}$ and $e_{01}$ ECH generators for $\partial X_{\Omega^\pm}.$
\end{lemma}

\begin{proof}
We just check that $\phi(e_{10}^+)=e_{10}^-$ for $J_0$ as described below, as the other equality is completely analogous. 

Let $\overline{X}$ denote the completion of the associated cobordism from $\partial X_{\Omega^+}$ to $\partial X_{\Omega^-}.$ Let $J_0$ be the standard complex structure on $X_{\Omega^+}\setminus (X_{\Omega^-})^\circ$ away from small neighborhoods of the boundary components, extended as in Prop 4.1 in \cite{Munteanu} to the completed cobordism $\overline{X}$ so that there is a trivial holomorphic cylinder $C_1$ in the moduli space $\mathcal{M}_{J_0}(e_{10}^+,e_{10}^-)$ (given by the $z_2$-coordinate axis in the cobordism region). That this cylinder persists for small generic perturbations of $J_0$ follows from the following automatic transversality result.
\begin{proposition}(Lemma 4.1 in \cite{CylHom})
Let $u$ be an immersed $J$-holomorphic curve in $X$ with ends at nondegenerate Reeb orbits. Suppose that $2g-2+h_+<\text{ind}(u)$, where $g$ is the genus of domain of $u$ and $h_+$ is the number of ends of $u$ at positively hyperbolic Reeb orbits. Then $u$ is cut out transversely. 
\end{proposition}

The trivial cylinders considered have genus zero with only elliptic ends, so the hypotheses are met and they are cut out transversely. Hence, for generic $J$ that is $C^\infty$-close to $J_0$, there still exists a cylinder from $e_{10}^+$ to $e_{10}^-$ that is cut out transversely.  

We check that $\mathcal{M}_J(e_{10}^+,e_{01}^-)$ is empty. Suppose for contradiction that there existed a holomorphic curve $C_2$ in this moduli space. Then, by the union bound we have $$I(C_1\cup C_2)\geq I(C_1)+I(C_2)=0.$$

On the other hand,

$$I(C_1\cup C_2)=I(e_{10}^2)-I(e_{10}e_{01})=4-6=-2.$$

Since chain-level cobordism maps on ECH count holomorphic curves, it follows that there is no $e_{01}^-$ term in $\phi(e_{10}^+).$ Also, $\phi$ induces an isomorphism on ECH, and $e_{10}^+$ represents a nonzero ECH class for $X_{\Omega^+},$ so $\phi(e_{10}^+)$ represents a nonzero ECH class for $X_{\Omega^-}.$ The only possibility left is $\phi(e_{10}^+)=e_{10}^-.$
\end{proof}

We are now ready to study some examples of isotopies of symplectic embeddings.
Let $r:\C^2\to\C^2$ be the reflection $r(z,w)=(w,z)$. Let $X_{\Omega^+}$ and $X_{\Omega^-}$ be toric domains such that $X_{\Omega^-}\cup (r(X_{\Omega^-}))\subseteq X_{\Omega^+}.$ Then we have two symplectic embeddings from $X_{\Omega^-}$ into $X_{\Omega^+}$: the standard inclusion $\iota$ as well as $r\circ\iota.$ We consider the question of when these two embeddings are isotopic to each other by a path of symplectic embeddings. 
Just as we described the cobordism map associated to $\iota$ on index 2 generators above in Lemma 6.3, we do the same for $r\circ\iota$ 

\begin{lemma}
Suppose that $X_{\Omega^{\pm}}$ are convex toric domains with $r(X_{\Omega^-})\subset\ X_{\Omega^+}.$ Let $\phi:ECC(\partial X_{\Omega^+})\to ECC(\partial X_{\Omega^-})$ be an associated cobordism map. Then for suitably chosen admissible $J$, we have $\phi(e_{10}^+)=e_{01}^-$ and $\phi(e_{01}^+)=e_{10}^-$
\end{lemma}

\begin{proof}
This follows from applying Lemma 6.3, substituting $r(X_{\Omega^-})$ for $X_{\Omega^-}.$ Since the $e_{10}$ orbit for $X_{\Omega}$ is set-theoretically the same as the $e_{01}$ orbit for $r(X_{\Omega^-})$ and vice-versa, the result follows. 
\end{proof}
We say that $X_{\Omega^-}$ and $X_{\Omega^+}$ are \textit{separated by a ball} if there exists a real number $R>0$ such that $X_{\Omega^-}\subseteq B(R)\subseteq X_{\Omega^+}$.

\begin{lemma}
Let $X_{\Omega^+}, X_{\Omega^-},r,\iota$ be as above. If $X_{\Omega^+}$ and $X_{\Omega^-}$ are separated by a ball, then $\iota$ and $r\circ\iota$ are isotopic by a path of symplectic embeddings.
\end{lemma}

\begin{proof}
Note that the transformation $r(z,w)=(w,z)$ is unitary. Since $U(2)$ is path-connected, there is a path of unitary transformations from the identity map on $\C^2$ to $r.$ Since unitary transformations preserve the unit ball and are symplectic, the restriction of such a path to the domain $X_{\Omega^-}$ yields the desired isotopy. 
\end{proof}

In some cases, being separated by a ball is also a necessary condition.

\begin{proof}[Proof of Theorem \ref{thm:reflection-isotopy}]
Suppose that $X_{\Omega^-}$ and $X_{\Omega^+}$ are convex toric domains. Then $\iota$ and $r\circ\iota$ are isotopic by a path of symplectic embeddings if and only if $X_{\Omega^+}$ and $X_{\Omega^-}$ are separated by a ball.

We already know that being separated by a ball is a sufficient condition.
For the other direction, suppose that there is a path of symplectic embeddings connecting $\iota$ to $r\circ\iota.$ For $i\in\{0,1\}$, let $$\phi_i:ECC_*(\partial X_{\Omega^+})\to ECC_*(\partial X_{\Omega^-})$$ be cobordism maps associated to $\iota$ and $r\circ\iota$ respectively. From Lemma 6.3, we have $\phi_0(e_{10}^+)=e_{10}^-$ and $\phi_0(e_{01}^+)=e_{01}^-,$ for $J$ as considered there. Also, for the moment, we shall assume that the almost complex structure on the completed cobordism associated to $r\circ\iota$ is $C^\infty$-close to the standard one, and in this case, we have $\phi_1(e_{10}^+)=e_{01}^-$ and $\phi_1(e_{01}^+)=e_{10}^-$ by Lemma 6.5. 
Since the embeddings $\iota$ and $r\circ\iota$ are isotopic, their corresponding cobordism maps are chain homotopic by Proposition 6.2. Let $K:ECC_*(\partial X_{\Omega^+})\to ECC_{*+1}(\partial X_{\Omega^-})$ be as given by Prop 6.2.  

\begin{equation}K\partial^++\partial^- K=\phi_1-\phi_0.\end{equation}

Also, $\partial e_{10}^+=\partial e_{01}^+=0$ since there are no convex generators of index 1. So, if we apply both sides of the chain homotopy equation (6.1) to $e_{10}$ we find that $\partial^- K(e_{10}^+)=e_{01}^--e_{10}^-.$ In particular, $K(e_{10}^+)\neq 0,$ so we must have $K(e_{10}^+)=h_{11}^-,$ as the righthand side is the only convex generator of index 3. Similarly, $K(e_{01}^+)=h_{11}^-.$ 
Since $K$ preserves the action filtration, we have that $$\mathcal{A}_{X_{\Omega^-}}(h_{11}^-)\leq\text{min} \left(\mathcal{A}_{X_{\Omega^+}}(e_{10}^+),\mathcal{A}_{X_{\Omega^+}}(e_{01}^+)\right).$$

Put $R=\mathcal{A}_{X_{\Omega^-}}(h_{11}^-).$ From the previous inequality and the convexity of the domains $X_{\Omega^\pm},$ it follows that $X_{\Omega^-}\subseteq B(R)\subseteq X_{\Omega^+}.$
If instead the almost complex structure $J_1$ associated to the embedding $r\circ\iota$ is not close to the standard one, it is possible that $\phi_1(e_{01}^+)=e_{01}^-$ and $\phi_1^+(e_{10})=e_{10}^-$. In this case, the above argument fails, as equation (6.1) does not rule out $K(e_{10}^+)=K(e_{01}^+)=0$. However, we may argue similarly. The space of almost complex structures on a completed cobordism is connected, so we may choose a (generic) path from the standard one $J_0$ to $J_1$. Fixing the embedding from $r\circ\iota$ and the contact form on $X_{\Omega^-}$, as in Prop 6.2 in \cite{Calabi}, we have a chain homotopy $H$ from this varying almost complex structure:

$$\phi_1^{J_1}-\phi_1^{J_0}=\partial^- H+H\partial^+,$$

where $\phi_1^{J_i}$ is a cobordism map associated to $J_i$ and $$H:ECC_*(\partial X_{\Omega^+})\to ECC_{*+1}(\partial X_{\Omega^-}).$$ 
Applying the same argument as above to this chain homotopy equation instead yields the same result.

\end{proof}

\begin{remark}
This result is equivalent to the corresponding statement for when $X_{\Omega^\pm}$ are both polydisks: given $\text{max}(a,b)\leq\text{min}(c,d)$, the inclusion $P(a,b)\hookrightarrow P(c,d)$ is isotopic to its reflection if and only if $a+b\leq\text{min}(c,d).$ This is equivalent because replacing $X_{\Omega^-}$ with the inscribed polydisk $P(a,b)$ where $(a,b)$ is a point on $\Omega^-$ with tangent slope of $-1$ and similarly replacing $X_{\Omega^+}$ with a circumscribed polydisk do not change any of the actions of the Reeb orbits considered.
\end{remark}

\begin{remark}This result is already known for inclusions into a square polydisk. Given positive real numbers $a,b,c$ with $\text{max}(a,b)\leq c$, it is shown in \cite{App} that $P(a,b)\hookrightarrow P(c,c)$ is isotopic to its reflection if and only if $a+b\leq c.$
\end{remark}
\appendix
\section{Contributions from Intersection Positivity}
In this appendix, we will see how we can further obstruct the existence of $J$-holomorphic curves using intersection positivity.

Throughout this section, let $X$ denote a cobordism from $\partial X_{\Omega^+}$ to $\coprod_{i=1}^m \partial X_{\Omega_i},$ of boundaries of convex toric domains, perturbed as in Section 4 so that the simple Reeb orbits of action smaller than some fixed, sufficiently large, bound $L$ are the familiar $e_{ab}$ and $h_{ab}.$ Let $J$ be a generic admissible almost complex structure on the completed cobordism $\overline{X}.$ Given orbit sets $\alpha$ and $\beta$ in the positive and negative domains, respectively, let $\mathcal{M}^J(\alpha,\beta)$ denote the moduli space of $J$-holomorphic currents in $\overline{X}$ positively asymptotic to $\alpha$ and negatively asymptotic to $\beta.$ All Reeb orbits considered in this section are assumed to have action smaller than $L$.

We first see how we can use intersection positivity to obtain more constraints on the convex generators $\Lambda^1,\ldots,\Lambda^m$ that may appear in the statement of Theorem \ref{thm:main-convex}. 

\begin{proposition}
Suppose that $C$ is a somewhere-injective holomorphic curve in the completed cobordism $\overline{X}$ with positive ends corresponding to the convex generator $\Lambda'$ and negative ends corresponding to the convex generators $\Lambda^1\ldots\Lambda^m.$ Similarly, suppose $\tilde{C}$ has ends corresponding to the convex generators $\tilde{\Lambda}'$ and $\tilde{\Lambda}^1,\ldots,\tilde{\Lambda}^m.$ If $C$ and $\tilde{C}$ have no components in common, then the following inequality holds

$$\sum_{i=1}^m \left(\text{Area}(\Lambda^i\tilde{\Lambda}^i)-\text{Area}(\Lambda^i)-\text{Area}(\tilde{\Lambda}^i)\right)+\sum_{i=1}^m m(\Lambda^i,\tilde{\Lambda}^i)\leq \text{Area}(\Lambda'\tilde{\Lambda}')-\text{Area}(\Lambda')-\text{Area}(\tilde{\Lambda}'),$$

where $m(\Lambda^j,\tilde{\Lambda}^j)$ is the sum over all elliptic orbits $e_{ab}$ of the common multiplicity of $e_{ab}$ in these generators. For example, $m(e_{10}^2e_{21},e_{10}^3e_{21}e_{31})=2+1=3.$
\end{proposition}

\begin{proof}
We recall from \cite{Concave} that the relative intersection number $Q_\tau$ may be defined on a pair $(C,C')$ by 

$$Q_\tau(C,\tilde{C})=C\cdot \tilde{C}-\ell_\tau(C,\tilde{C}).$$

Since this is bilinear in the relative homology classes $[C],[\tilde{C}]$, it may be evaluated in terms of $Q_\tau(C)$ and $Q_\tau(\tilde{C})$ separately by a polarization identity:

\begin{align*}Q_\tau(C,\tilde{C})&=\frac{1}{2}\left(Q_\tau(C\cup \tilde{C})-Q_\tau(C)-Q_\tau(\tilde{C})\right)\\
&=\left(\text{Area}(\Lambda'\tilde{\Lambda}')-\sum_{i=1}^m \text{Area}(\Lambda^i\tilde{\Lambda}^i)\right)\\ &\phantom{{}={}}-\left(\text{Area}(\Lambda')-\sum_{i=1}^m \text{Area}(\Lambda^i)\right)-\left(\text{Area}(\tilde{\Lambda}')-\sum_{i=1}^m \text{Area}(\tilde{\Lambda}^i)\right).
\end{align*}

By the bounds for linking numbers in \cite{LectureNotes}, each term $\ell_\tau(\zeta^+,\tilde{\zeta}^+)$ is $\leq 0,$ and $\ell_\tau(\zeta^-,\tilde{\zeta}^-)$ is greater than or equal to the common minimum winding number of $\zeta^-,\tilde{\zeta}^-$ around their associated orbit $\gamma^-.$ 

For a fixed simple elliptic Reeb orbit $\gamma^-$ in $\partial X_{\Omega_j}$, let $q_1,\ldots,q_r$ and $\tilde{q}_1,\ldots,\tilde{q}_s$ denote the multiplicities of the ends of $C$ and $\tilde{C}$, respectively, at $\gamma^-$. In particular, $q:=\sum_{i=1}^r q_i$ equals the exponent of $\gamma$ in the factorization of $\Lambda^j,$ and likewise $\tilde{q}:=\sum_{i=1}^s \tilde{q}_i$ for $\tilde{\Lambda}^j$.

Then, the linking number of $C$ and $\tilde{C}$ at $\gamma^-$ is greater than or equal to $\sum_{i,j}\text{min}(q_i,\tilde{q}_j).$ It is straightforward to verify 

$$\sum_{i,j}\text{min}(q_i,\tilde{q}_j)\geq \text{min}(q,\tilde{q}),$$

and the result follows by summing over all Reeb orbits, and using the nonnegativity of the signed intersection number $C\cdot\tilde{C}.$
\end{proof}

We give some examples of how this proposition may be used to obstruct holomorphic curves. First, item (g) in Theorem \ref{thm:main-convex} may be strengthened to the following:

\begin{itemize}
\item (g$'$) For a convex generator $\Lambda$, let $m_{ab}(\Lambda)$ denote the exponent of $e_{ab}$ in $\Lambda$.  For all $j\in \{1,\ldots,n\}$ and $i\in\{1\ldots,m\}$, the following inequality holds:
$$\text{min}\left(x(\Lambda_j'),y(\Lambda_j'))\geq\text{min}(x(\Lambda_j^i)+\text{min}(1,m_{01}(\Lambda_j^i)),y(\Lambda_j^i)+\text{min}(1,m_{10}(\Lambda_j^i))\right).$$
\end{itemize}

To see this, apply Proposition A.1 to the curve $C$ corresponding to the $j$-th column, and $C'$ a curve with a single positive end asymptotic to $e_{10}$ and a single negative end asymptotic to a Reeb orbit in $\partial X_{\Omega_i}.$

If this negative end is $e_{10},$ then since $\text{Area}(e_{10}\Lambda)=y(\Lambda)+\text{Area}(\Lambda)$ for any convex generator, the above proposition gives

$$y(\Lambda_j^i)+\text{min}(1,m_{10}(\Lambda_j^i))\leq y(\Lambda_j').$$

If instead the negative end is $e_{01}$, then the same inequality holds, with the lefthand side replaced with $x(\Lambda_j^i)+\text{min}(1,m_{01}(\Lambda_j^i)).$ Repeating this with a curve $C'$ having a positive end at $e_{01}$ gives the result.

Frequently, Proposition A.1 can obstruct curves that otherwise would not be ruled out by Theorem \ref{thm:main-convex}.

\begin{example}\label{ex:three-copies-obstruction}
Suppose we have a symplectic embedding $\coprod_{i=1}^3 X_{\Omega_i}\hookrightarrow B(c)$ of convex toric domains into a ball. 

We use the above proposition to verify that there cannot exist a $J$-holomorphic current $ \mathcal{C}$ in the completed cobordism that is described by the column $\begin{pmatrix}e_{11}^5\\e_{10}^2e_{11}\\e_{10}^3e_{01}\\e_{10}^3e_{01}\end{pmatrix},$ even though such a current is not ruled out by Theorem \ref{thm:main-convex}.

Let $\mathcal{C}'$ be a current with positive ends given by $e_{11}^3$ and index split \newline$I(\Lambda^1)=I(\Lambda^2)=I(\Lambda^3)=6$, as in the proof of Theorem \ref{thm:main-convex}. It must be connected, as a direct search reveals that there does not exist a factorization $$\begin{pmatrix}e_{11}\\ \Lambda_1^1\\ \Lambda_1^2\\ \Lambda^3_1\end{pmatrix}\begin{pmatrix}e_{11}^2\\ \Lambda_2^1\\ \Lambda^2_2\\ \Lambda^3_2\end{pmatrix}=\begin{pmatrix}e_{11}^3\\ \Lambda^1\\ \Lambda^2\\ \Lambda^3\end{pmatrix}$$ or $$\begin{pmatrix}e_{11}\\ \Lambda^1_1\\ \Lambda_1^2\\ \Lambda_1^3\end{pmatrix}\begin{pmatrix}e_{11}\\ \Lambda_2^1\\ \Lambda_2^2\\ \Lambda_2^3\end{pmatrix}\begin{pmatrix}e_{11}\\ \Lambda_3^1\\ \Lambda_3^2\\ \Lambda_3^3\end{pmatrix}=\begin{pmatrix}e_{11}^3\\ \Lambda^1\\ \Lambda^2\\ \Lambda^3\end{pmatrix},$$ where each column satisfies the index constraint and the righthand side of either equation has the prescribed index split.

Note that $\mathcal{C}$ cannot have $\mathcal{C}'$ as a component because the generator $e_{10}^2e_{11}$ in the first negative component does not have any index $6$ factor. Hence Proposition A.1 applies.

If, for example, the curve $\mathcal{C}'$ is given by the column $\begin{pmatrix}e_{11}^3\\ e_{10}^3\\e_{10}^3\\ e_{10}^3\end{pmatrix}$ then the inequality of Proposition A.1 becomes

\begin{equation}\label{eq:example-bound}\sum_i \text{min}(3,m_{10}(\Lambda^i))+3\sum_i y(\Lambda^i)\leq 15\end{equation}

since $\text{Area}(e_{10}^3\Lambda)=\text{Area}(\Lambda)+3y(\Lambda)$ for any convex generator $\Lambda$.
This rules out the existence of a current $\mathcal{C}$ as the lefthand side of (\ref{eq:example-bound}) is $(2+3+3)+3(1+1+1)=17>15.$ 

The other possibilities for the ends of $\mathcal{C}'$ yield a similar contradiction.
\end{example}

\begin{remark} We only used a ball as the positive domain to guarantee that $e_{11}^3$ and $e_{11}^5$ were minimal, but if we use the fact that any convex generator with all edges labeled ``e'' represents a nonzero ECH class, then it may be replaced by a general convex toric domain $X_{\Omega}.$
\end{remark}

\begin{remark}
An interesting combinatorial question is whether obstructing $J$-holomorphic curves with this criterion may lead to sharper bounds than those obtained from Theorem 3.1 alone. For example, when considering the problem of symplectically embedding $\coprod_{i=1}^3 P(a,1)$ into $B(c)$, when $a$ is sufficiently large relative to $c$ (but still small enough that an embedding is possible), then for any fixed $d>0$, the only possibility for a column with first component of $e_{11}^d$ and index split $d^2+3d=2r_1+2r_2+2r_3$ is $\begin{pmatrix}e_{11}^d\\ e_{10}^{r_1}\\ e_{10}^{r_2}\\ e_{10}^{r_3}\end{pmatrix}$, and so if $\begin{pmatrix}e_{11}^f\\ \Lambda^1\\ \Lambda^2\\ \Lambda^2\end{pmatrix}$ is any other column generated by Theorem 3.1, an inequality similar to (7.1) must hold:

$$\sum_i\text{min}(r_i,m_{10}(\Lambda^i))+\sum_i r_iy(\Lambda^i)\leq d\cdot f.$$
\end{remark}

We may also apply intersection positivity to show that many moduli spaces can contain at most one $J$-holomorphic curve satisfying some conditions.

\begin{proposition}\label{thm:uniqueness-index-zero}
Let $C\in \mathcal{M}^J(\alpha,\beta)$ satisfy the following conditions. 
\begin{itemize}
\item All components of $C$ have genus zero.
\item $C$ is somewhere-injective.
\item The ECH index of $C$ is zero.
\item All orbits in $\alpha$ and $\beta$ are elliptic.
\end{itemize}

Then, $C$ is the unique index zero and somewhere-injective element of $\mathcal{M}^J(\alpha,\beta).$
\end{proposition}

Before proving this, we need a couple of preliminary results.

First, we have the following computation.
\begin{lemma}\label{lem:area-lemma}
Let $C=\sum_{j=1}^k C_j\in \mathcal{M}^J(\alpha,\beta)$ be somewhere-injective, with all orbits in $\alpha$ and $\beta$ elliptic. Let $\Lambda^+$ and $\Lambda^1,\ldots,\Lambda^m$ be the convex generators corresponding to $\alpha$ and $\beta$. Then,

\begin{align*}
&\sum_{i=1}^m x(\Lambda^i)+y(\Lambda^i)\\=&x(\Lambda^+)+y(\Lambda^+) + \frac{1}{2}\left(m(\Lambda^+)-\sum_{i=1}^m m(\Lambda^i)\right)-k \\ &+ \sum_{j=1}^k g(C_j)+\frac{1}{2}\left(n(C)-w_\tau(C)-I(C)\right)+\delta(C),
\end{align*}
where $n(C)$ is the number of ends of $C$ (or equivalently, the number of punctures of the domain Riemann surface).
\end{lemma}

\begin{proof}
This is an application of the relative adjunction formula $$c_\tau(C)=\chi(C)+Q_\tau(C)+w_\tau(C)-2\delta(C).$$ These terms may be evaluated as in our proof of the genus bound:
\begin{itemize}
\item $c_\tau(C)=x(\Lambda^+)+y(\Lambda^+)-\sum_{i=1}^m x(\Lambda^i)+y(\Lambda^i)$
\item $Q_\tau(C)=2\text{Area}(\Lambda^+)-2\sum_{i=1}^m \text{Area}(\Lambda^i).$ 
\item $\chi(C)=2-2\sum_{j=1}^k g(C_j)-n(C).$ 
\end{itemize}

Also, whenever $\Lambda$ is a convex generator with all edges labeled ``$e$'', Pick's formula gives 

\begin{align*}
2\text{Area}(\Lambda)&=2L(\Lambda)-2-(x(\Lambda)+y(\Lambda)+m(\Lambda))\\ &=I(\Lambda)-x(\Lambda)-y(\Lambda)-m(\Lambda).
\end{align*}

Substituting into the relative adjunction formula, and using the condition $$I(C)=I(\Lambda^+)-\sum_{i=1}^m I(\Lambda^i)$$ gives the desired result.
\end{proof}

Next, we need to know about the partition data of these holomorphic curves.

\begin{lemma}
Suppose that $C\in \mathcal{M}^J(\alpha,\beta)$ has index zero and no multiply covered components. Then, for every simple orbit $\alpha_i$ appearing in $\alpha,$ each end of $C$ at a cover of $\alpha_i$ has multiplicity $1.$ Also, no two ends of $C$ are asymptotic to covers of the same simple negative Reeb orbit $\beta_i$ appearing in $\beta$. 
\end{lemma}

\begin{proof}
This follows from the equality condition for the ``partition conditions'' described in Section 3.9 of \cite{LectureNotes}. 
\end{proof}

We are now ready to prove Proposition A.5.

\begin{proof}
Suppose for contradiction that $C'$ is another current in the same moduli space $\mathcal{M}^J(\alpha,\beta)$. If $C$ and $C'$ have any components in common, we delete those from both curves and replace $C$ and $C'$ with the remaining sub-currents.

Let $k$ be the number of components of $C$.

The assumptions on $C$ imply that $I(C)=\delta(C)=g(C_j)=0$ in the statement of Lemma \ref{lem:area-lemma} and so we obtain:

\begin{align*}
&\left(\sum_{i=1}^m x(\Lambda^i)+y(\Lambda^i)+\frac{1}{2}m(\Lambda^i)\right)\\
=&\left(x(\Lambda^+)+y(\Lambda^+)+\frac{1}{2}m(\Lambda^+)\right)-k+\frac{1}{2}\left(n(C)-w_\tau(C)\right).
\end{align*}

Since $C'$ has ends at the same Reeb orbits with multiplicity as $C$, the lefthand side is equal to $Q_\tau(C,C)=Q_\tau(C),$ which we compute as

$$2\text{Area}(\Lambda^+)-2\sum_{i=1}^m \text{Area}(\Lambda^i)=\left(\sum_{i=1}^m x(\Lambda^i)+y(\Lambda^i)+m(\Lambda^i)\right)-x(\Lambda^+)-y(\Lambda^+)-m(\Lambda^+).$$

Put $m^+=m(\Lambda^+)$ and $m^-=\sum_{i=1}^m m(\Lambda^i).$ Then combining with the previous identity yields

\begin{align*}
C\cdot C'&=Q_\tau(C,C')+\ell_\tau(C,C')\\
&=-k+\ell_\tau(C,C')+\frac{1}{2}\left(n(C)-m^++m^--w_\tau(C)\right) \tag{*}
\end{align*}

where the first equality is from the definition of $Q_\tau$ in \cite{Quant}. We recall the definition 

$$\ell_\tau(C,C')=\sum_{\gamma^+}\sum_{\zeta^+,\tilde{\zeta}^+}\ell_\tau(\zeta^+,\tilde{\zeta}^+)-\sum_{\gamma^-}\sum_{\zeta^-,\tilde{\zeta}^-}\ell_\tau(\zeta^-,\tilde{\zeta}^-),$$

where $\gamma^{\pm},\tilde{\gamma}^{\pm}$ run over the set of positive/negative simple Reeb orbits at which $C,C'$ have ends at covers of, and $\zeta^{\pm},\tilde{\zeta}^{\pm}$ are the (possibly empty) braids around $\gamma^{\pm}$ corresponding to the ends of $C$ and $C'.$

The righthand side of (*) is $-k$ plus a sum of contributions from each simple Reeb orbit in $\alpha$ and $\beta$, separately. We check that each Reeb orbit contributes nonpositively to the sum. Let $\gamma$ be a simple Reeb orbit appearing in $\alpha$ with multiplicity $q$. The linking bound for positive orbits in \cite{LectureNotes} shows that it contributes nonpositively to the $\ell_\tau$ term in the sum. Its contribution to the rest of the righthand side is $$\frac{1}{2}\left(n(\gamma)-m(\gamma)-(m(\gamma)-n(\gamma))\right)=0$$ by the partition data in the previous lemma. 

If instead $\gamma$ is a simple orbit appearing in $\beta$, with multiplicity $q$, let $q=q_1+\ldots+q_r$ be the partition associated to the ends of $C'$ at $\gamma$. Then the linking bound tells us that the contribution to $\ell_\tau$ is at most $$-\sum_{i=1}^r\text{min}(q,q_i)=-q,$$ since $C$ has a single end at $\gamma.$ The contribution to the rest of the righthand side from $\gamma$ is $$\frac{1}{2}(1+q-(1-q))=q,$$ so altogether, $\gamma$ contributes nonpositively in this case too. Hence we find $C\cdot C'\leq -1$, which contradicts intersection positivity.
\end{proof}

\end{document}